\documentclass[11pt]{amsart}

\usepackage{enumerate}
\usepackage{esint} %%%% txfonts conflicted with other packages' font 

\usepackage[USenglish]{babel}
\usepackage[T1]{fontenc} 
\usepackage[utf8,latin1]{inputenc}

\usepackage{amsmath,amsthm,amssymb,amsfonts,mathrsfs}
\usepackage{dsfont}
\usepackage{mathtools}

\usepackage{soul}

\usepackage{nicematrix}
\usepackage{lmodern}
\usepackage[all]{xy}

\usepackage[bookmarks=true]{hyperref}
\usepackage{xcolor}
\hypersetup{
    colorlinks,
    linkcolor={red!50!black},
    citecolor={blue!50!black},
    urlcolor={blue!80!black},
}

\usepackage[bottom=3cm, top=3cm, left=3.5cm, right=3.5cm]{geometry}

\usepackage{caption}

\mathtoolsset{showonlyrefs}  % uncomment to tag only the referred equations (post-production)

\newcommand{\sr}{sub-Riemannian }

\newcommand{\vol}{{\rm vol}_\g}

\newcommand{\eps}{\varepsilon}
\newcommand{\di}{\mathsf d}
\newcommand{\g}{\mathsf g}
\newcommand{\X}{\text{X}}
\newcommand{\de}{\text{d}}

\DeclareMathOperator{\Geo}{Geo}
\newcommand{\p}{\mathtt p}

\newcommand{\diam}{\text{diam}}
\newcommand{\cut}{\text{Cut}}
\newcommand{\bW}{\mathbb W}

\newcommand{\jacobi}{{\sf J}}
\DeclareMathOperator{\trace}{tr}

\newcommand{\rcd}{{\sf RCD}}
\newcommand{\cd}{{\sf CD}}
\newcommand{\mcp}{{\sf MCP}}

\newcommand{\bm}{{\sf BM}}

\newcommand{\cD}{\mathcal{D}}
\newcommand{\cE}{\mathcal{E}}

\newcommand{\cM}{\mathcal{M}}

\newcommand{\cJ}{\mathcal{J}}

\newcommand{\wei}{\mathfrak m}

\newcommand{\Haus}{\mathscr{H}}

\newcommand{\Prob}{\mathscr{P}}

\newcommand{\Ric}{\text{Ric}}
\newcommand{\Ricn}{\Ric^{N,\wei}}

\newcommand{\spt}{\text{spt}}

\newcommand{\Leb}{\mathscr{L}}

\newcommand{\R}{\mathbb{R}}

\newcommand{\N}{\mathbb{N}}

\newcommand{\tand}{\quad\text{ and }\quad}

\newcommand{\suchthat}{\ensuremath{\ : \ }} % such that inside, for example, the sets definition

\newcommand{\rie}{\text{Riem}}
\newcommand{\tr}{\text{tr}}
\newcommand{\id}{\text{Id}}
\newcommand{\tB}{\text{B}}

\newcommand{\lin}{\text{lin}}

\theoremstyle{plain}
\newtheorem{thm}{Theorem}[section]

\newtheorem{lem}[thm]{Lemma}
\newtheorem*{lem*}{Lemma}
\newtheorem{prop}[thm]{Proposition}

\theoremstyle{definition}
\newtheorem{defn}[thm]{Definition}

\theoremstyle{remark}
\newtheorem{rmk}[thm]{Remark}
\newtheorem*{notation}{Notation}

\author[M. Magnabosco]{Mattia Magnabosco}
\author[L. Portinale]{Lorenzo Portinale}
\author[T. Rossi]{Tommaso Rossi}
\address{Institut f\"ur Angewandte Mathematik, Universit\"at Bonn, Bonn, Germany}
\email{\href{mailto:magnabosco@iam.uni-bonn.de}{magnabosco@iam.uni-bonn.de}}
\email{\href{mailto:rossi@iam.uni-bonn.de} {rossi@iam.uni-bonn.de}}
\email{\href{mailto:portinale@iam.uni-bonn.de}{portinale@iam.uni-bonn.de}}

\title[\texorpdfstring{$\bm(K,N)$}{BM(K,N)} implies \texorpdfstring{$\cd(K,N)$}{CD(K,N)} in weighted Riemannian manifolds]{The Brunn--Minkowski inequality implies the CD condition in weighted Riemannian manifolds}

\date{\today}

\begin{document}

\begin{abstract}
The curvature dimension condition $\cd(K,N)$, pioneered by Sturm and Lott--Villani in \cite{Sturm:2006-I,Sturm:2006-II,Lott-Villani:2009}, is a synthetic notion of having curvature bounded below and dimension bounded above, in the non-smooth setting. This 
condition implies a suitable generalization of the Brunn--Minkowski inequality, denoted $\bm(K,N)$. In this paper, we address the converse implication in the setting of weighted Riemannian manifolds, proving that $\bm(K,N)$ is in fact equivalent to $\cd(K,N)$. Our result allows to characterize the curvature dimension condition without using neither the optimal transport nor the differential structure of the manifold.
\end{abstract}

\maketitle

% \tableofcontents

\section{Introduction}
    \label{sec:Introduction}

In their seminal papers \cite{Sturm:2006-I,Sturm:2006-II,Lott-Villani:2009}, Sturm and Lott--Villani introduced a synthetic notion of
curvature dimension bounds, in the non-smooth setting of metric measure spaces, usually denoted by $\cd(K,N)$, with $K\in\R$, $N\in [1,\infty]$. They observed that, in a (weighted) Riemannian manifold, the \emph{differential} notion of having Ricci curvature bounded below, and dimension bounded above, can be equivalently characterized in terms of a convexity property of the R\'enyi entropy functional, along Wasserstein geodesics. In particular, the latter property relies on the theory of optimal transport and does not require the smooth underlying structure. Therefore, it can be taken as a \emph{synthetic} definition, for a metric measure space, to have a curvature dimension bound. 

Among its many merits, the $\cd(K,N)$ condition is sufficient to deduce geometric and functional inequalities that hold in the smooth setting. An example is the so-called Brunn--Minkowski inequality, whose classical version in $\R^n$ (see e.g. \cite{Gardner:2002}) states that
\begin{equation}
    \mathscr{L}^n\left((1-t) A+ t B\right)^\frac{1}{n} \geq (1-t) \mathscr{L}^n(A)^\frac{1}{n} +  t \mathscr{L}^n(B)^\frac{1}{n},\quad\forall t\in [0,1]\,,
\end{equation}
for every two nonempty compact sets $A,B\subset\R^n$. 
%{\blue Equivalently, it states that the map $A \mapsto \mathscr{L}^n(A)^\frac{1}{n}$ is concave}. 
In \cite{Sturm:2006-II}, Sturm proved that a $\cd(K,N)$ space supports a generalized version of the Brunn--Minkowski inequality, denoted $\bm(K,N)$, replacing the Minkowski sum of $A$ and $B$ with the set of midpoints and employing the so-called \emph{distortion coefficients}, see Definition \ref{def:bruno} for further details. 

In this paper, we address the converse implication: indeed there is a general belief in the optimal transport community that the Brunn--Minkowski inequality $\bm(K,N)$ is sufficient to deduce the $\cd(K,N)$ condition. This work provides a first positive partial answer to this problem in the setting of weighted Riemannian manifolds. 

\begin{thm}[$\bm(K,N)\Rightarrow\cd(K,N)$]
\label{thm:intro}
    Let $(M,\g)$ be a complete Riemannian manifold of dimension $n$, endowed with the reference measure $\wei= e^{-V} \vol$, where $V \in C^2(M)$. Suppose that the metric measure space $(M,\di_\g, \wei)$ satisfies $\bm(K,N)$ for some $K\in\R$ and $N>1$. Then, it is a $\cd(K,N)$ space (and in particular the two conditions are equivalent).
    %Then the metric measure space$(M,\di_g, \wei)$ is a $\cd(K,N)$ space if and only if it satisfies $\bm(K,N)$.
\end{thm}

As mentioned before, in a weighted Riemannian manifolds $(M,\di_\g,e^{-V}\vol)$, the $\cd(K,N)$ condition is equivalent to the lower Ricci bound $\Ricn\geq K\g$, see Definition \ref{def:ricci_tensor}. Our result allows us to characterize both conditions, without using neither the optimal transport nor the differential structure of the manifold. Moreover, as a corollary of Theorem \ref{thm:intro} and in light of \cite{MR2644212}, $\bm(K,N)$ is equivalent to a modified Borell--Brascamp--Lieb inequality. see \cite[Definition~1.1]{MR2644212} for the precise definition. 

\subsection*{Relations with the \texorpdfstring{$\mcp$}{MCP} condition}
In \cite{Ohta:2007}, the author introduced the so-called \emph{measure contraction property}, $\mcp(K,N)$ for short, for a general metric measure space. This condition for a \emph{non-weighted} Riemannian manifold is equivalent to having the (standard) Ricci tensor bounded below, see \cite[Theorem~3.2]{Ohta:2007}. However, in general, the $\mcp(K,N)$ condition is strictly weaker than the $\cd(K,N)$ condition, and this is also the case for weighted Riemannian manifolds. Theorem \ref{thm:intro} confirms that $\bm(K,N)$ is much closer to the $\cd(K,N)$ condition than the $\mcp(K,N)$ condition is. 

We mention that in \cite{MR3935035}, in ideal \sr manifolds, a different version of the Brunn--Minkowski inequality has been studied. When $K=0$, this turns out to be equivalent to the $\mcp(0,N)$ condition, thus it is strictly weaker than $\bm(0,N)$.

\subsection*{Strategy of the proof of Theorem \ref{thm:intro}} The idea of the proof is to deduce the differential characterization of the $\cd(K,N)$ condition, arguing by contradiction. Thus, we assume there exists $v_0\in T_{x_0}M$, with $x_0\in M$ such that
\begin{equation}
\label{s}
    \Ricn_{x_0}(v_0,v_0)<K\|v_0\|^2,
\end{equation}
and then find two subsets $A,B\subset M$ contradicting $\bm(K,N)$. The first step is to build a suitable optimal transport map $T$ moving the mass in a neighborhood of $x_0$ in the direction $v_0$, see Section \ref{sec:choice_kant}. The second step is to estimate the infinitesimal volume distortion around the geodesic $\gamma(t)=\exp_{x_0}(tv_0)$, joining $x_0$ and $T(x_0)$, cf. Proposition \ref{prop:upper_bound_hatD}. By means of a comparison principle for ordinary differential equations (cf. Lemma~\ref{lem:1d_comparison}), the condition \eqref{s} implies that
\begin{equation}
\label{eq:volume_distortion_intro}
    \wei\left(T_{\frac12}(A)\right)^{\frac1N}< \tau_{K,N}^{\left(\frac12\right)}(\Theta(A,B))\left(\wei(A)^\frac1N+\wei(B)^\frac1N\right), 
\end{equation}
where $T_{\frac12}$ is the interpolating optimal transport map, $A\subset M$ is any sufficiently small neighborhood of $x_0$ and $B:=T(A)$. For the technical definitions of the distortion coefficients $\tau_{K,N}^{(t)}(\cdot)$ and $\Theta(A,B)$, see Definitions \ref{def:distortion} and \ref{def:bruno}. The final, and most challenging, step is to compare the measure of $T_{\frac12}(A)$ with the measure of $M_{\frac12}(A,B)$, the set of midpoints between $A$ and $B$, cf. Definition \ref{def:tmidpoints}. This is done through a careful analysis of the behavior of the map $T$ and choosing as $A$ a specific cube oriented according to the Riemann curvature tensor at $x_0$,  cf. Section~\ref{sec:choice_basis}. We then obtain
\begin{equation}
\label{eq:comparison_intro}
    \wei\left(M_{\frac12}(A,B)\right)\approx \wei(T_{\frac12}(A)),
\end{equation}
which, together with \eqref{eq:volume_distortion_intro}, gives a contradiction with $\bm(K,N)$. The relation \eqref{eq:comparison_intro} is made rigorous in Proposition \ref{prop:control}, whose proof is based on the linearization of the map $T_{\frac12}$. Remarkably, a second-order expansion capturing the local geometry of the manifold (involving in particular the Riemann curvature tensor at $x_0$) is needed.

\subsection*{Open problems}
It would be relevant to extend Theorem \ref{thm:intro} to general \emph{essentially non-branching} metric measure spaces. Indeed, on the one hand this would produce an equivalent characterization oh the $\cd$ condition without the need of optimal transport. On the other hand, it would provide an alternative proof of the \emph{globalization theorem}, cf. \cite{MR4309491}, using \cite[Theorem~1.2]{MR3608721}.
In \cite{Magnabosco-Portinale-Rossi:2022b}, we prove that in essentially non-branching metric measure spaces, the $\cd(K,N)$ condition is in fact equivalent to a stronger version of $\bm(K,N)$, denominated \emph{strong Brunn--Minkowski condition}. However, the equivalence between the two, at this level of generality, seems to be out of reach with the techniques developed up to now. Nonetheless, one may hope to adapt our strategy either to the setting of Finsler manifolds or to the one of $\rcd$ spaces, where a second-order calculus is available.

\subsection*{Acknowledgments} 
% M.M. and T.R. European Research Council (ERC) under the program ERC-AdG RicciBounds, grant agreement No. 694405. 
L.P. acknowledges support by the Hausdorff Center for Mathematics in
Bonn. T.R. acknowledges support from the Deutsche Forschungsgemeinschaft (DFG, German Research Foundation) through the collaborative research centre ``The mathematics of emerging effects'' (CRC 
1060, Project-ID 211504053). The authors are grateful to Prof. Karl-Theodor Sturm for inspiring discussions on the topic.

\section{Preliminaries}
    \label{sec:preliminaries}
%\begin{enumerate}
    %\item Riem. manifold, Riemann and Ricci tensor, modified Ricci tensor.
    %\item Equivalence between CD(K,N) and bounds on the modified Ricci tensor.
    %\item Distortion coefficients, modified BM inequality.
%\end{enumerate}
In this section, we introduce the general framework of interest, recalling basic facts about optimal transport and Riemannian manifolds. Moreover, we present the Brunn--Minkowski inequality and state our main result.

\subsection{Riemannian manifolds}
Let $(M,\g)$ be a Riemannian manifold of dimension $n \in \N$. Let $TM$ denote the tangent bundle of $M$, and for $x \in M$,  $T_xM$ the tangent space at $x$. For simplicity of notation, whenever it does not create ambiguity, we write for $x \in M$ and $v,w \in T_x M$
\begin{align}
    \langle v, w \rangle := \g_x(v,w)
        \tand 
    \| v \|^2 := \g_x(v,v) 
        \, .
\end{align}
Let $\di_\g$ the Riemannian distance associated with $\g$, defined by length--minization procedure, and we say that $(M,\g)$ is complete if $(M,\di_\g)$ is a complete metric space. Furthermore, for $x\in M$, we denote by $\exp_x: T_xM \to M$ the exponential map, i.e. $\exp_x(v) = \gamma_{x,v}(1)$, where $\gamma_{x,v}$ denotes the geodesic on $M$ such that $\gamma_{x,v}(0)=0$ and $\dot \gamma_{x,v}(0)=v \in T_xM$ (whenever it is well-defined for $t=1$).
Denote by $\nabla$ the associated Levi-Civita connection on $(M,\g)$ and, for $X \in TM$, by $\nabla_X$ the covariant derivative along the vector field $X$. Then the Riemann curvature tensor is defined as
\begin{align}
    \label{eq:def_rie}
    \rie(X,Y) := \nabla_X \nabla_Y - \nabla_Y \nabla_X + \nabla_{[X,Y]}
        \, , \quad 
    X,Y \in TM \, .
\end{align}
The Ricci tensor is obtained by taking suitable traces of the Riemann tensor (more precisely it is the trace of the sectional curvature tensor). In details, one has that
\begin{align}
    \Ric: TM \times TM \to \R 
        \, , \quad 
    \Ric(Y,Z) := \tr 
    \big(
        X \mapsto \rie(X,Y)Z 
    \big)
        \, .
\end{align}
Both the Riemann and the Ricci tensor naturally appears in the study of the volume deformation along geodesics, see Section~\ref{sec:infinitesimal}.

In particular, the Ricci tensor is closely related to convexity properties of entropy functionals along the geodesics of optimal transport. In the framework of weighted Riemannian manifolds a similar role is played by a modified version of the Ricci tensor, which depends on a dimensional parameter $N \in \N$ and a reference measure $\wei$.

\begin{defn}[Modified Ricci tensor]
\label{def:ricci_tensor}
    Let $(M,\g)$ be a Riemannian manifold, let $V\in C^2(M)$ and consider the measure $\wei=e^{-V}\vol$ on $M$, where $\vol$ is the Riemannian measure. Fix $N \geq n$. Then the $(N,\wei)$-modified Ricci tensor is given by
    \begin{align}
        \Ricn := \Ric + \nabla^2 V - 
        \frac
            {\nabla V \otimes \nabla V}
            {N-n}
        \, . 
    \end{align}
    Here $\nabla^2 V$ denotes the Hessian of $V$, suitably identified to a bilinear form.
    With the convention that $0 \cdot \infty =0 $, if $N=n$ then necessarily $\nabla V =0$, and thus $\Ricn = \Ric$.
\end{defn}

\subsection{Optimal transport and curvature}
%Def of OT, optimal maps, definition of CD, equivalence.
Let $(\X,\mathsf{d},\wei)$ be a metric measure space, i.e.  $(\X,\mathsf{d})$ is a complete and separable metric space and $\wei$ is a non-negative Borel measure on $\X$, finite on bounded sets. Denote by $C([0, 1], \X)$ the space of continuous curves from $[0, 1]$ to $\X$, and define the $t$-\textit{evaluation map} as $e_t:C([0,1],\X) \to \X$; $e_t(\gamma):= \gamma(t)$, for $\gamma \in C([0,1],\X)$. A curve $\gamma\in C([0, 1], \X)$ is called \textit{geodesic} if 
\begin{equation}
    \di(\gamma(s), \gamma(t)) = |t-s| \cdot \di(\gamma(0), \gamma(1)) \quad \text{for every }s,t\in[0,1]\,.
\end{equation}
We denote by $\Geo(\X) \subset C([0,1],\X)$ the space of constant speed geodesics in $(\X,\di)$. The metric space $(\X,\di)$ is said to be \textit{geodesic} if every two points are connected by a curve in $\Geo(\X)$. Note that any complete Riemannian manifold is a geodesic metric space.

Denote by $\mathscr{P}(\X)$ the set of all Borel probability measures on $\X$ and by $\mathscr{P}_2(\X) \subset \mathscr{P}(\X)$ the set of all probability measures with finite second moment. The 2-Wasserstein distance $\bW_2$ is a distance on the space $\mathscr{P}_2(\X)$ defined by
\begin{align}\label{eq:W2}
    \bW_2^2(\mu,\nu):=\inf_{\pi\in{\text{Adm}}(\mu,\nu)}
        \int_{\X\times \X}\di(x,y)^2  \de\pi(x,y) \, , 
\end{align}
for $\mu$, $\nu \in \Prob_2(M)$, where $\text{Adm}(\mu,\nu):=\{\pi\in\Prob(\X\times\X): (\p_1)_\# \pi= \mu,\,(\p_2)_\# \pi= \nu\}$ is the set of admissible plans. Here $\p_i: \X \times \X \to \X$ denotes the projection on the $i$-th factor. The infimum in \eqref{eq:W2} is always attained, the admissible plans realizing it are called \emph{optimal transport plans} and are denoted by $\text{Opt}(\mu,\nu)$. Whenever a plan $\pi \in \text{Opt}(\mu,\nu)$ is induced by a map $T:\X \to \X$ if $\pi=(\text{id},T)_\# \mu$, we say that $T$ is an \emph{optimal transport map}.  
It turns out that $\bW_2$ defines a complete and separable distance on $\Prob_2(\X)$ and, moreover $(\Prob_2(\X),\bW_2)$ is geodesic if and only if $(X,\di)$ is.

In this paper, we work in the setting of weighted Riemannian manifolds, namely considering the metric measure space $(M,\di_\g,\wei)$, where $\wei= e^{-V} \vol$.
In this framework, whenever $\mu$ is absolutely continuous with respect to $\vol$, then the optimal plan is unique and induced by an optimal transport map $T: M \to M$, see \cite{Benamou-Brenier:2000}, \cite{McCann:2001}. Moreover, $T$ is driven by the gradient of the so-called \textit{Kantorovich potential} $-\psi: M \to \R$ via the exponential map as
\begin{align*}
    T(x) = \exp_x (\nabla \psi(x))
        \, , \quad \forall x \in M \, ,
\end{align*}
where $\psi$ is a semiconvex function, cf. \cite[Definition~10.10]{Villani:2009}.

Since the seminal works of Sturm \cite{Sturm:2006-I,Sturm:2006-II} and Lott--Villani \cite{Lott-Villani:2009}, it is known that lower bounds on the (modified) Ricci curvature tensor can be recast in a synthetic way in terms of a suitable entropy convexity property. The latter is formulated in terms of the R\'enyi entropy functional and the distortion coefficients. 

\begin{defn}[R\'enyi entropy functional] Let $(\X,\di,\wei)$ be a metric measure space and fix $N>1$. The $N$-\emph{R\'enyi entropy functional} on $\Prob_2(\X)$ is defined as 
\begin{equation}
    \mathcal{E}_N (\mu)=  -\int_{\X} \rho(x)^{1-\frac{1}{N}} \, \de\wei(x) \qquad \forall \mu \in \Prob_2(\X)\,,
\end{equation}
where $\rho$ is the density of the absolutely continuous part of $\mu$, with respect to $\wei$.
\end{defn} 

% {\blue The distortion coefficients depends on both the curvature and dimension parameter $K$,$N$, as they naturally appear when computing the volume distortion along geodesics in the model manifold of dimension $N$ and constant curvature $K$. } 
\begin{defn}[Distortion coefficients]   \label{def:distortion}
    For every $K \in \R$, $N >0$, we define for $\theta \geq 0$
\begin{align}
    \sigma_{K,N}^{(t)} (\theta) := 
    \left\{
        \begin{array}{ll}   
            +\infty 
                &  N \pi^2\leq K \theta^2 \\ \displaystyle 
            \frac
            {\sin(t \alpha)}
            {\sin(\alpha)} 
                & 0<K \theta^2 < N \pi^2 \\  
            t 
                & K=0  \\
                    \displaystyle 
            \frac
            {\sinh(t \alpha)}
            {\sinh(\alpha)}  & K<0
        \end{array}
    \right. 
        \, , \quad 
    \alpha := \theta \sqrt{\frac{|K|}{N}} 
        \, ,
\end{align}
 while for $K \in \R$ and $N >1$ we introduce the \textit{distortion coefficients} for $t \in [0,1]$ as
 \begin{equation}
     \tau_{K, N}^{(t)}(\theta):=t^{\frac{1}{N}} \sigma_{K, N-1}^{(t)}(\theta)^{1-\frac{1}{N}} \, .
 \end{equation}
\end{defn}

% \begin{rmk}[Monotonicity]   \label{rem:monotonicity_tau}
%     By construction, for $t \in [0,1]$, $K \in \R$, $N>1$, the map $\theta \mapsto \tau_{K, N}^{(t)}(\theta)$ is increasing for $K>0$, constant for $K=0$, and decreasing for $K<0$. Moreover, the map $(K,N)\mapsto \tau_{K,N}^{(t)}$ is non-decreasing in $K$ and non-increasing in $N$.
% \end{rmk}

\begin{defn}[$\cd(K,N)$ space]
Given $K \in \R$ and $N >1$, a metric measure space $(\X,\di,\wei)$ is said to be a $\cd(K,N)$ \emph{space} (or to satisfy the $\cd(K,N)$ condition) if for every pair of measures $\mu_0=\rho_0\wei$, $\mu_1= \rho_1 \wei \in \Prob_2(\X)$, there exists $\eta \in \Prob(\Geo(\X))$ such that $\mu_t :=(e_t)_\# \eta=: \mu_t \ll \wei$ is $\bW_2$-geodesic from $\mu_0$ and $\mu_1$ which satisfies the following inequality, for every $N'\geq N$ and every $t \in [0,1]$:
\begin{equation}\label{eq:CDcond}
    \cE_{N'}(\mu_t) 
        \leq 
    -\int_{\X \times \X} 
    \Big[ 
        \tau^{(1-t)}_{K,N'} \big( \di(x,y) \big)     
            \rho_{0}(x)^{-\frac{1}{N'}} 
                +
        \tau^{(t)}_{K,N'} \big(\di(x,y) \big) 
            \rho_{1}(y)^{-\frac{1}{N'}} 
    \Big]   
        \de \pi(x,y) \, ,
\end{equation}
where $\pi= (e_0,e_1)_\# \eta$ is the optimal plan between $\mu_0$ and $\mu_1$ induced by $\eta$.
\end{defn}

Note that in the case $K=0$, the distortion coefficients are linear in $t$, and the  $\cd$ condition simply becomes convexity of the R\'enyi entropy functional along $\bW_2$-geodesics.
We end this section stating an equivalence result between the $\cd$ condition and a Ricci bound for weighted Riemannian manifolds, which plays a crucial role in the sequel.

\begin{thm}[Equivalence theorem \cite{Sturm:2006-II}, \cite{Lott-Villani:2009}]
\label{thm:equiv_CD_RicN}
    Let $K \in \R$ and $N> 1$ and let $(M,\g)$ be a complete Riemannian manifold of dimension $n$. For $N\geq n$, the metric measure space $(M,\di_g, e^{-V} \emph{vol}_\g)$ is a $\cd(K,N)$ space if and only if 
    \begin{align}   \label{eq:lowerbound_modRicci}
        \emph{\Ric}^{N,\wei} \geq K 
            \qquad 
        (
        \text{thus: }
            \emph{\Ric}^{N,\wei}_x(v,v) \geq K \| v \|^2 
                \, , \quad 
            \forall x \in M , \, v \in T_xM 
        )
           \, .
    \end{align}
\end{thm}

\subsection{Brunn--Minkowski inequality}
We now introduce a generalized version of the Brunn--Minkowski inequality, tailored to a curvature  parameter and a dimensional one. As proven by Sturm in \cite{Sturm:2006-II}, this is a consequence of the $\cd$ condition. 

\begin{defn}[$t$-midpoints]\label{def:tmidpoints}
    Let $(\X,\di)$ be a metric space. Let $A,B \subset \X$ be two Borel subsets. Then for $t\in (0,1)$, we defined the set of $t$-midpoints between $A$ and $B$ as
    \begin{align}
        M_t(A,B) := 
        \{
            x \in \X \, : \, x = \gamma(t)
                \, , \, 
            \gamma \in \Geo(\X)
                \, , \, 
            \gamma(0) \in A \, , 
                \ \text{and} \  
            \gamma(1) \in B
        \} 
            \, .
    \end{align}
\end{defn}

\begin{defn}[Brunn--Minkowski inequality]
\label{def:bruno}
    Let $K \in \R$, $N >1$. Then we say that a metric measure space $(\X,\di, \wei)$ satisfies the Brunn--Minkowski inequality $\bm(K,N)$ if for every $A,B \subset \spt(\wei)$ Borel subsets, $t \in (0,1)$, we have
        \begin{align}   \label{eq:bm}
            \wei \big(M_t(A,B)\big) \big)^ \frac{1}{N} 
                \geq 
            \tau_{K,N}^{(1-t)} (\Theta(A,B)) \cdot \wei(A)^ \frac{1}{N} 
                + 
            \tau_{K,N}^{(t)} (\Theta(A,B)) \cdot \wei(B)^ \frac{1}{N}
                \, ,
        \end{align}
        where 
    \begin{align}   \label{eq:def_Theta}
        \Theta(A,B):=
        \left\{
        \begin{array}{ll}\displaystyle    \inf_{x \in A,\, y \in B} \di(x, y) 
                & \text { if } K \geq 0 \, , \\ 
        \displaystyle
            \sup _{x \in A,\, y \in B} \di(x, y) 
                & \text { if } K<0 \, .
        \end{array}
        \right. 
    \end{align}
%\item If $N=\infty$, then for every $t \in (0,1)$
\end{defn}

\begin{rmk}
\label{rmk:measurability_Mt}
In general the set $M_t(A,B)$ is not Borel measurable, even if the sets $A$ and $B$ are Borel. Therefore, when $M_t(A,B)$ is not measurable, the left-hand side of \eqref{eq:bm} has to be intended with the outer measure $\bar \wei$ associated with $\wei$, in place of $\wei$ itself.
\end{rmk}

\begin{prop}[{\cite[Proposition~2.1]{Sturm:2006-II}}]
    Let $K \in \R$, $N >1$ and let $(\X,\di, \wei)$ be a metric measure space satisfying the $\cd(K,N)$ condition. Then, $(\X,\di, \wei)$ satisfies $\bm(K,N)$.
\end{prop}

A remarkable feature of this result lies in the sharp dependence of the Brunn--Minkowski inequality on the curvature exponent $K\in\R$ and the dimensional parameter $N>1$. In a weighted Riemannian manifold, using the equivalence result of Theorem \ref{thm:equiv_CD_RicN}, we prove that the \emph{sharp} Brunn--Minkowski inequality is enough to deduce the $\cd$ condition, with the \emph{same constants}. We are in position to state our main result, which is a rephrasing of Theorem \ref{thm:intro}, in view of Theorem \ref{thm:equiv_CD_RicN}.

\begin{thm}[$\bm(K,N)\Rightarrow\Ricn\geq K$]
\label{thm:main}
    Let $(M,\g)$ be a complete Riemannian manifold of dimension $n$, endowed with the reference measure $\wei= e^{-V} \vol$, where $V \in C^2(M)$. Suppose that the metric measure space $(M,\di_\g, \wei)$ satisfies $\bm(K,N)$ for some $K\in\R$ and $N>1$. Then $\Ricn\geq K$, in the sense of \eqref{eq:lowerbound_modRicci}.
    %Then the metric measure space$(M,\di_g, \wei)$ is a $\cd(K,N)$ space if and only if it satisfies $\bm(K,N)$.
\end{thm}

\begin{rmk}
In \cite[Theorem~1.7]{Sturm:2006-II} and \cite[Theorem~1.2]{Ohta:2009}, the authors prove the implication $\cd(K,N)\Rightarrow\Ricn\geq K$. Their argument is based on an inequality involving the $t$-midpoints, which would imply Theorem~\ref{thm:main}. However, they do not address the problem of comparing the latter set with the support of the interpolating optimal transport, cf. Proposition~\ref{prop:control}, which is the crucial and most challenging step. We stress that our choice of sets is more involved, precisely with the aim of obtaining a better control on the $t$-midpoints, which was lacking in the previous constructions. Finally, we point out that their argument works verbatim replacing the $t$-midpoints with the support of the interpolating optimal transport, nonetheless it does not provide an effective strategy for Theorem~\ref{thm:main}.
\end{rmk}

\section{Brunn--Minkowski implies \texorpdfstring{$\cd$}{CD}}
    \label{sec:proof}
Our strategy to prove Theorem \ref{thm:main} is to proceed by contradiction: let $(M,\di_\g,\wei)$ support $\bm(K,N)$, and assume there exist $\delta>0$, $x_0\in M$, and $v_0\in T_{x_0}M$, with $\|v_0\|=1$, such that 
\begin{equation}    \label{eq:ric_absurd}
    \Ricn_{x_0}(v_0, v_0)<(K-3\delta) \, .
\end{equation}
%In particular, there exists $\delta>0$ such that \eqref{eq:upper_bound} holds true with $K-2\delta$.
More precisely, taking $\lambda$ small enough, we can assume
\begin{gather}
\label{eq:geodesic_v0_small}
    \gamma^\lambda(t) := \exp_{x_0}(t \lambda v_0)
         \  \text{is a geodesic on $(M,g)$ from $x_0$ to $\gamma^\lambda(1)$} \, , \\
\label{eq:upper_bound}
    \Ricn_{\gamma^\lambda(t)}(\dot \gamma^\lambda(t) ,  \dot \gamma^\lambda(t))<(K-2\delta)\lambda^2 \, , 
        \ \ \text{for every $t \in [0,1]$} \, .
\end{gather}
The idea is to exploit the fact that the generalized Ricci tensor controls the infinitesimal distortion of volumes around $\gamma^\lambda$, to build explicitly two sets $A$, a neighborhood of $x_0$, and $B$, a neighborhood of $\gamma^\lambda(1)$, contradicting the Brunn--Minkowski inequality $\bm(K,N)$. 

\subsection{Infinitesimal volume distortion}
    \label{sec:infinitesimal}
 In this section, we recall general facts regarding infinitesimal volume distortion around a geodesic starting at any point $x\in M$, see \cite[Chapter~14]{Villani:2009}. In the next section, we specialize the results to the geodesic $\gamma^\lambda$.

To capture the infinitesimal volume distortion given by the (generalized) Ricci tensor, consider a transport map 
\begin{equation}
    T\colon M\rightarrow M,\qquad T(z)=\exp_x(\nabla\psi(z)) \, ,
\end{equation}
where $\psi\in C^3_c(M)$. In analogous way, the transport map interpolating the identity and $T$ is given by $T_t(z)=\exp_x(t\nabla\psi(z))$.
Fix an orthonormal basis $\{e_1, \dots, e_n \}$ of $T_xM$ and let $Q_\eps(x) \subset M$ be the image, via the exponential map $\exp_x$, of the cube of size $\eps>0$ centered at the point $x$ with sides given by the $e_i$.
Define the (weighted) \emph{Jacobian determinant} by
\begin{equation}
\cJ_x(t):=
\lim_{\eps\to 0}
    \frac{\wei(T_t(Q_\eps(x)))}{\wei(Q_\eps(x))}
=   
    \frac{e^{-V(T_t(x))}}{e^{-V(x)}}
    \lim_{\eps\to 0}
        \frac{\vol(T_t(Q_\eps(x)))}{\vol(Q_\eps(x))}
=
    \frac{e^{-V(T_t(x))}}{e^{-V(x)}}\tilde \cJ_x(t) \, .
\end{equation} 
It turns out that the function $\tilde\cJ_x$ has a useful geometric interpretation. Let $\gamma\colon [0,1]\rightarrow M$ be a geodesic connecting $x$ with $T(x)$ (thus with velocity $\nabla \psi(x)$ in $t=0$) and let $\{e_1(t),\ldots,e_n(t)\}$ be a parallel orthonormal frame along $\gamma$ with $e_i(0)=e_i$. For any $i=1,\ldots,n$, consider the Jacobi vector field $J_i(t) := J_i(t,x)$ solving
\begin{equation}
\label{eq:jacobi}
    \ddot{J}(t)+\rie(\dot\gamma_x(t),J(t))\dot\gamma_x(t)=0 \, , 
\qquad 
    J(0)=e_i \, ,
        \quad 
    \dot J(0)=\nabla^2\psi(x)e_i \, ,
\end{equation}
where $\rie$ is defined in \eqref{eq:def_rie}. For $x\in M$ and $t\in [0,1]$, we define the $n \times n$ matrix  $\jacobi_x(t):=(J_1(t,x)|\ldots|J_n(t,x))$. The relation with $\tilde \cJ_x(t)$ is then given by the identity
\begin{equation}
    \tilde \cJ_x(t)=\det \jacobi_x(t) \, . 
\end{equation}
Due to our interest in the Brunn--Minkowski inequality, the natural quantity to look at is given by $\cD_x(t)=\cJ_x(t)^{1/N}$. Differentiating the determinant and using \eqref{eq:jacobi}, one can prove that $\cD_x(t)$ satisfies a Riccati--type equation involving the generalized Ricci tensor, given by
\begin{equation}    \label{eq:ODE_D(t)}
    -N\frac{\cD_x''(t)}{\cD_x(t)}=
        \Ricn_{\gamma(t)}(\dot\gamma(t), \dot\gamma(t))
            + \cE_x(t) \, ,
    \qquad
        \forall \, t\in [0,1]\, , \, x\in M \, . 
\end{equation} 
The remainder term $\cE_x(t)$
% \footnote{Note that the functions $\cD_x(\cdot)$, $U_x(\cdot)$ and $\cE_x(\cdot)$ implicitly depend on the velocity of $\gamma$. With a slight abuse of notation, we omit this dependence whenever is clear from the context.}
depends on the transport map and is defined as follows: set $U_x(t)=\partial_t\jacobi(t,x) \jacobi^{-1}(t,x)$, then we have an explicit expression given by
\begin{equation}
    \cE_x(t) =
    \left\|
        U_x(t)-\frac{\trace U_x(t)}{n}\,{\rm Id}
    \right\|_{{\rm HS}}^2
        +
    \frac{n}{N(N-n)}
    \left|
        \frac{N-n}{n}\trace U_x(t)
            +
        \g_{\gamma(t)} 
        \big( 
            \dot\gamma(t),\nabla V(\gamma(t))
        \big)
    \right|^2 , 
\end{equation}
where $\| \cdot \|_{\rm HS}$ denotes the Hilbert--Schmidt norm of a matrix.  
\begin{rmk}     \label{rmk:homogeneity_E}
    The term $\cE_x(t)$ enjoys good behavior under reparametrization of the curve $\gamma$. Indeed, if we see the map $\cE_x$ as a function of $t$ and $v:=\dot \gamma(0)$, then we have that 
\begin{equation}
    \cE_x(t) 
        = \cE_x(t,v) 
        =\|v\|^2 \cE_x
            \left(
                \| v\|t,\frac{ v}{\| v\|}
            \right)
        \, , \qquad \forall t\in [0,1] \, .
\end{equation} 
In particular, for any given $\lambda\in [0,1]$, if we consider the curve $\gamma^\lambda(s) := \gamma(\lambda s)$ for $s \in [0,1]$, then the corresponding functional $ \cE_{x,\lambda}$, obtained via the Jacobi vector fields along the reparametrized curve $\gamma^\lambda$, satisfies
\begin{align}   \label{eq:hat_cE}
    \cE_{x,\lambda}(s) 
        = \cE_x( \lambda s , \lambda v )
        = \lambda^2 \|v\|^2 \cE_{x}
            \left(
                \lambda^2\|v\|s,\frac{v}{\|v\|}
            \right)
        \, , \qquad \forall \, s\in [0,1] \, .
\end{align}
\end{rmk}

\begin{notation}
From now on, once a curve $\gamma$ is fixed, whenever we consider a reparametrization $\gamma^\lambda(s):=\gamma(\lambda s)$, all the quantities defined in this section, and associated with $\gamma^\lambda$, are denoted with a subscript $\lambda$. 
\end{notation}

\subsection{Choice of the Kantorovich potential}
    \label{sec:choice_kant}
In order to exploit the upper bound in \eqref{eq:upper_bound}, we would like to control the volume distortion along the direction of $v_0$, thus we choose a Kantorovich potential to suitably \emph{drive} the transport along that direction. In particular, fix $\psi\in C^3_c(M)$ such that
\begin{equation}    \label{eq:conditions_psi}
    \begin{cases}
        \nabla\psi(x_0)=v_0 \, , \\
        \nabla^2\psi(x_0)=\alpha_0 \,{\rm Id} \, , \\
        \Delta\psi(x_0)=n\,\alpha_0=-\frac{n}{N-n}\g(\nabla V(x_0),v_0) \, .
    \end{cases}
\end{equation}
In addition, define $\psi_\lambda=\lambda\psi$, for any $\lambda\in [0,1]$. Note that $\|\psi_\lambda\|_{C^2}\sim\lambda$, hence, for $\lambda$ sufficiently small, we can apply \cite[Theorem 13.5]{Villani:2009} and deduce that $\psi_\lambda$ is $\frac{d^2}{2}$-convex. As a consequence, the following map
\begin{equation}
\label{eq:transport_map}
    T^\lambda(x) := \exp_x( \nabla \psi_\lambda(x) )= \exp_x( \lambda \nabla \psi(x) )
\end{equation}
is optimal, by \cite[Theorem 8]{McCann:2001}. We also denote by $T_t^\lambda(x):=\exp_x(t\nabla \psi_\lambda(x))$, for any $x\in M$ and $t\in [0,1]$, the interpolating optimal map between the identity and $T_1^\lambda=T^\lambda$.

The unique geodesic joining $x_0$ and $T^1(x_0)$ is exactly given by $\gamma(t):=\exp_{x_0}(t v_0)$ and, by definition, $U_{x_0}(0)=\nabla^2\psi(x_0)$. Thus, we have that
% \begin{equation}
% \label{eq:vanishing_error}
$
    \cE_{x_0}(0)=0
$.
This choice of $\psi$ is closely related to the Bochner inequality and its equivalence to lower bounds on the modified Ricci tensor, see e.g. \cite[Theorem~14.8]{Villani:2009}.
% \end{equation}
In particular, since $\cE_{x_0}(\cdot)$ is a continuous function on $[0,1]$, there exists $\bar\lambda\in (0,1]$ such that 
\begin{equation}
\label{eq:refined_ineq}
    \cE_{x_0}(t)\leq \delta , \qquad\forall t\in[0,\bar\lambda] \, .
\end{equation}
% Thanks to Remark~\ref{rmk:homogeneity_E}, this can be equivalently recast as 
% \begin{equation}    \label{eq:E_geq_rho}
%     \cE_{x_0}
%     \left(
%         \|v_0\|t,\frac{v_0}{\|v_0\|}
%     \right)
%     \leq \rho 
%         \, , \qquad \forall t\in [0,t_0] \, .
% \end{equation} 
Now, for any $\lambda\leq\bar\lambda$ we reparametrize $\gamma$ on the interval $[0,\lambda]$ obtaining $\gamma^\lambda$ as in \eqref{eq:geodesic_v0_small}. Then, according to the notation introduced in the previous section, denoting $ \cD_{x_0,\lambda}$ the associated (power of the) Jacobian determinant and with $\cE_{x_0,\lambda}$ the remainder as in Remark~\ref{rmk:homogeneity_E}, from \eqref{eq:upper_bound} and \eqref{eq:hat_cE}, with $v=\lambda v_0$ we obtain that
\begin{equation}    
\label{eq:ODI_hatD0}
    -N\frac{\cD_{x_0,\lambda}''(s)}{\cD_{x_0,\lambda}(s)}=
        \Ricn_{\gamma^\lambda(s)}
            (\dot{\gamma}^\lambda(s) , \dot{\gamma}^\lambda(s) )
                +
            \cE_{x_0,\lambda}(s)
        < (K-\delta) \lambda^2
            \, ,\qquad \forall s \in (0,1) \, .
\end{equation}
% In what follows, we consider the geodesic $\gamma^\lambda$, satisfying \eqref{eq:geodesic_v0_small}, \eqref{eq:upper_bound}, as well as \eqref{eq:ODI_hatD0} (where we omit the hat for convenience), i.e.
% \begin{align}\label{eq:ODI_D}
%     \cD_{x_0}^{''}(t) + \frac{K-\delta}{N} \lambda^2 \cD_{x_0}(t) > 0
%         \, ,\qquad \forall t \in (0,1) \, .
% \end{align}
The next step is to provide a suitable one-dimensional comparison result for $ \cD_{x_0,\lambda}$, as a solution of the ordinary differential inequality \eqref{eq:ODI_hatD0}.

\subsection{One-dimensional comparison}
The following lemma is in the same spirit of
\cite[Theorem~14.28]{Villani:2009}, although concerning the reverse inequality. 
\begin{lem} \label{lem:1d_comparison}
Let $\Lambda<\pi^2$ and $f\in C([0,1])\cap C^2(0,1)$, with $f\geq 0$. Then the following are equivalent:
\begin{enumerate}
\item $\ddot f+\Lambda f\geq 0$ in $(0,1)$;
\item For all $t,s_0,s_1\in [0,1]$, 
\begin{equation}    \label{eq:comparison_II}
    f((1-t)s_0+t s_1)\leq \sigma^{(1-t)}(|s_0-s_1|)f(s_0)+\sigma^{(t)}(|s_0-s_1|)f(s_1) \, ,
\end{equation}
where $\sigma^{(t)}(\cdot)=\sigma^{(t)}_{\Lambda,1}(\cdot)$, according to the notation in Definition~\ref{def:distortion}.
%If $\Lambda=\pi^2$, $f(t) = c\sin(\pi t)$ for some $c\geq 0$; finally if $\Lambda>\pi^2$, then $f=0$.
\end{enumerate}
\end{lem}

\begin{proof}
$(1)\Rightarrow (2)$. \  
Let  $\Lambda \neq 0$, the case $\Lambda =0$ being trivial. Set $f(t):=f((1-t)s_0+t s_1)$, then by assumption it satisfies
\begin{equation}
    \label{eq:additional_assumption}
    f''(t)+ \bar \Lambda f(t)\geq 0,\qquad\forall\,t\in [0,1]\, , \quad \bar \Lambda := \Lambda |s_1-s_0|^2 \, .
\end{equation}
Let $g$ be the right-hand side of \eqref{eq:comparison_II}, as a function of $t$. In particular, $g$ solves the same equation of $f$ with an equal sign, i.e. $g'' + \bar \Lambda g =0$, and $g(0)=f(0)$, $g(1) = f(1)$. Let $w:[0,1] \to \R_+$ be any solution to the problem
\begin{align}
    w''(t) + \bar \Lambda w(t) > 0
        \tand 
    w > 0 \quad \text{on} \quad (0,1) \, .
\end{align}
If $\Lambda>0$, then we can choose $w\equiv 1$; if $\Lambda<0$, we can choose e.g. $w(t) = \exp(t \sqrt{-\bar \Lambda+1})$. Then for every $a \in \R_+$, we can define $f_a:= f + a w$ and let $g_a$ be the solution to $g_a'' + \bar \Lambda g_a = 0$ with $g_a(0) = f_a(0)$, $g_a(1)=f_a(1)$. We note that $f_a > 0$ on $(0,1)$ by construction and $f_a \to f$, $g_a \to g$ uniformly, as $a \to 0$. Moreover, we also have that
\begin{align}   \label{eq:strict}
    f_a''(t)+ \bar \Lambda f_a(t) > 0,\qquad\forall\,t\in [0,1] \, .
\end{align} 
Therefore, without loss of generality, we can assume $f,g>0$ and $f/g \in C^2$ in $(0,1)$ and \eqref{eq:strict} is satisfied with $f$ instead of $f_a$. In order to prove \eqref{eq:comparison_II}, we shall prove that $f/g$ attains its maximum in $\{0,1\}$. By contradiction, let $t_0 \in (0,1)$ be a maximum for $f/g$, hence $(f/g)'(t_0) =0$ and $(f/g)''(t_0) \leq 0$. An elementary computation shows that
\begin{align}
    \left( \frac{f}{g} \right)'' = 
    \frac
        {f''+\bar \Lambda f}{g}
    - \frac{f}{g^2}
        (g'' + \bar \Lambda g)
    -2 \frac{g'}g
        \left(  \frac{f}g  \right)'
     .
\end{align}
Evaluating at $t=t_0$, since $g>0$, and using the equation solved by $g$, we find that $f''(t_0)+\bar \Lambda f(t_0) \leq 0$, which is a contradiction.

\smallskip
\noindent 
$(2)\Rightarrow (1)$. \ 
Consider the Taylor expansion of $\sigma^{(\frac12)}(\cdot)$ at $\theta=0$, namely
\begin{equation}
\label{eq:taylor_exp_sigma}
\sigma^{\left(\frac{1}{2}\right)}(\theta)=\frac{1}{2}\left(1+\frac{\Lambda}{8}\theta^2\right)+o(\theta^3) \, .
\end{equation}
Analogously, we have the following Taylor expansion for $f$,
\begin{equation}
\label{eq:taylor_exp_f}
\frac{f(s_0)+f(s_1)}{2}=f\left(\frac{s_0+s_1}{2}\right)+\frac{1}{2}\ddot f\left(\frac{s_0+s_1}{2}\right)\left|\frac{s_1-s_0}{2}\right|^2+o(|s_0-s_1|^2) \, .
\end{equation}
Now, fix $t\in (0,1)$ and let $s_0,s_1\to t$ in such a way $t=\frac{s_0+s_1}{2}$. From \eqref{eq:taylor_exp_f}, we obtain
\begin{equation}
\frac{f(s_0)+f(s_1)}{2}=f(t)+\frac{1}{8}\ddot f(t)\left|s_1-s_0\right|^2+o(|s_0-s_1|^2) \, .
\end{equation}
Moreover, evaluating \eqref{eq:taylor_exp_sigma} at $\theta=|s_0-s_1|$, we rewrite 
\begin{equation}
\sigma^{\left(\frac{1}{2}\right)}(|s_0-s_1|)=\frac{1}{2}\left(1+\frac{\Lambda}{8}|s_0-s_1|^2\right)+o(|s_0-s_1|^3) \, .
\end{equation} 
Finally, putting all together we obtain that
\begin{equation}
\sigma^{\left(\frac{1}{2}\right)}(|s_0-s_1|)\left(f(s_0)+f(s_1)\right)-f(t)=\frac{1}{8}|s_0-s_1|^2\left(\Lambda f(t)+\ddot f(t)+o(1)\right)\geq 0  \,  .
\end{equation}
Taking the limit as $s_0,s_1\to t$ leads to the conclusion.
\end{proof} 

As a consequence of Lemma~\ref{lem:1d_comparison}, we prove the following upper bound. Recall the definition of $\bar\lambda\in (0,1]$ in \eqref{eq:refined_ineq}.

\begin{prop}    \label{prop:upper_bound_hatD}
    %  Assume that $v_0 \in T_{x_0} M$ is such that \eqref{eq:geodesic_v0_small}, \eqref{eq:upper_bound} are satisfied and such that $N^{-1} (K-\delta) \| v_0 \|^2 < \pi^2$. Set $w_0 := t_0 v_0$. Then we have that
    %  Set $w_0 := t_0 v_0$. 
     There exist $\lambda_1 \in (0,\bar\lambda)$ and $c>0$ not depending on $\bar\lambda$, such that, whenever $\lambda \leq \lambda_1$, we have that
    \begin{align}\label{eq:cpiccolo}
        \cD_{x_0,\lambda}\bigg( \frac12 \bigg) \leq \tau_{K,N}^{\left( \frac12 \right)}(\lambda)
            \left(
                \cD_{x_0,\lambda}(0) +  \cD_{x_0,\lambda}(1)  
            \right)    
             - c \lambda^2 \, .
    \end{align}
    % $\cD_{x_0,\lambda}$ is the volume distortion function associated with $\gamma^\lambda$.
    % The estimate for every s, if needed.
    % \begin{align}
    %     \hat \cD_{x_0}(s) \leq \tau_{K,N}^{(1-s)}(\| w_0 \|) \hat \cD_{x_0}(0) + \tau_{K,N}^{(s)}(\| w_0 \|) \hat \cD_{x_0}(1)  
    %         + C
    %     (
    %         - s(1-s)\| w_0\|^2 
    %         + o( \|w_0\|^2 ) 
    %     ) 
    % \end{align}
\end{prop}

\begin{proof}
Setting $f(t) := \cD_{x_0,\lambda}(t)$ and $\Lambda:= N^{-1} (K-\delta) \lambda^2$, we can choose $\rho \in \R_+$ small enough such that, for any $\lambda \leq \rho$, conditions \eqref{eq:geodesic_v0_small} and \eqref{eq:upper_bound} are satisfied and $\Lambda < \pi^2$. The estimate obtained in \eqref{eq:ODI_hatD0} shows that the hypothesis of Lemma~\ref{lem:1d_comparison} is satisfied. Therefore, using that $\sigma_{K-\delta,N}^{(s)}\leq \tau_{K-\delta,N}^{(s)}$, for any $s\in[0,1]$, we have 
 \begin{align}\label{eq:ineqDconv}
        \cD_{x_0,\lambda}(t) \leq \tau_{K-\delta,N}^{(1-t)}(\lambda)  \cD_{x_0,\lambda}(0) + \tau_{K-\delta,N}^{(t)}(\lambda)  \cD_{x_0,\lambda}(1)  
            \, , \qquad  \forall t \in [0,1] \, .
    \end{align}
% We already know from Remark~\ref{rem:monotonicity_tau} that $\tau_{K-\delta,N}^{(t)}(\lambda) \leq \tau_{K,N}^{(t)}(\lambda)$. In order to conclude the proof of the Proposition, we seek a more quantitative statement for $\tau_{K-\delta,N}^{(1-t)}(\lambda)$.
Recalling Definition~\ref{def:distortion}, we consider the Taylor expansion of $\tau_{K,N}^{(t)}(\lambda)$, as $\lambda\to 0$, obtaining 
% we have that for $\bar K \in \R$\todo{Va fatta per K<0 perooo. T: ho controllato e viene uguale, mettendo $|\bar K|$, per\`o in entrambi i casi mi sembra ci sia $\bar N$ a denominatore}, $\bar N>1$ 
% \begin{align}
%     \sigma_{\bar K,\bar N}^{(t)}(\theta) = t
%     \left(
%         1 + (1-t^2) \theta^2 \frac{\bar K}{\bar N-1}
%     \right)
%         + o(\theta^2) 
%     \, , \qquad  
%         \forall t \in [0,1] \, ,
% \end{align}
% where the $o(\cdot)$ only depends on $\bar K$, $N$, and  $\lim_{\theta \to 0} o(\theta)/\theta = 0$. As consequence, by definition of $\tau^{(t)}$ we infer that
\begin{align}
    \tau_{K,N}^{(t)}(\lambda) = t
    \left(
        1 + (1-t^2)  \frac{K}{6N} \lambda^2
    \right)
        + o(\lambda^2) 
    \, , \qquad  
        \forall t \in [0,1] \, .
\end{align}
Then, using this expansion, we can deduce that, as $\lambda\to 0$,
\begin{align}
\label{eq:exp_difference_tau}
    \tau_{K-\delta,N}^{(t)}(\lambda) = \tau_{K,N}^{(t)}(\lambda) - t (1-t^2) \frac{\delta}{6N}\lambda^2  + o(\lambda^2) \,,\qquad  
        \forall t \in [0,1] .
\end{align}
Combining \eqref{eq:ineqDconv} and \eqref{eq:exp_difference_tau} for $t=1/2$, and noting that $\cD_{x_0,\lambda}(0) +  \cD_{x_0,\lambda}(1)\geq \cD_{x_0,\lambda}(1)=1$, we obtain % Moreover, by continuity, $\cD_{x_0}(0) +  \cD_{x_0}(1)$ is uniformly bounded above provided that $\lambda$ is small enough.
\begin{align}
        \cD_{x_0,\lambda}\left(\frac{1}{2}\right) &\leq \tau_{K,N}^{(\frac12)}(\lambda)\left( \cD_{x_0,\lambda}(0) +  \cD_{x_0,\lambda}(1)\right)-\frac{\delta}{16N}  \lambda^2 + o (\lambda^2)\\
        &\leq \tau_{K,N}^{(\frac12)}(\lambda)\left( \cD_{x_0,\lambda}(0) +  \cD_{x_0,\lambda}(1)\right)-c  \lambda^2 ,
    \end{align}
where the last inequality holds definitely as $\lambda \to 0$, for a suitable positive constant $c>0$, independent on $\lambda$. This concludes the proof.
\end{proof}
% 
% In particular, under the same hypothesis of Proposition~\ref{prop:upper_bound_hatD}, we have that 
% \begin{align}
%     \hat \cD_{x_0}\left( \frac12 \right) \leq \tau_{K,N}^{\left( \frac12 \right)}(\| w_0 \|)
%         \left(
%             \hat \cD_{x_0}(0) + \hat \cD_{x_0}(1)  
%         \right)    
%         - C \| w_0\|^2 \, ,
% \end{align}
% where $C \in \R_+$ is a constant that only depends on $K$, $N$, and $\delta$.
% 
Proposition~\ref{prop:upper_bound_hatD} is a step forward towards the contradiction of the Brunn--Minkowski inequality. Indeed, on one side, $\cD_{x_0,\lambda}$ measures the infinitesimal volume distortion given by the transport map $T^\lambda$ around the geodesic $\gamma^\lambda$. On the other side, the inequality \eqref{eq:cpiccolo} goes in the opposite direction with respect to the Brunn--Minkowski inequality. The next step is to find an initial set $A$, as a suitable infinitesimal cube generated by an orthonormal basis $\{e_1, \dots, e_n\}$, such that the distortion steered by $\cD_{x_0,\lambda}$ allows to estimate the mass of the midpoints between $A$ and $T^\lambda(A)$.

% Intuitively, Proposition~\ref{prop:upper_bound_hatD} states that, if we look close enough ($\lambda \ll 1$) to the point $x_0$ where the modified Ricci curvature $\Ricn$ is bounded from above by $K-2\delta$, than we see that, asymptoptically for $\eps \to 0$, the strong Brunn--Minkowski inequality with parameters $(K,N)$ fails to hold. What we are left to do is to find a suitable choice for the orthonormal basis $\{e_1, \dots, e_n\}$ that allows us to find a contradiction to the actual Brunn--Minkowski inequality $\bm(K,N)$.

\subsection{The choice of the basis}
\label{sec:choice_basis}
For a fixed point $y \in M$, consider the squared distance function $\di_y^2(x) := \di_\g^2(x,y)$, for $x \in M$. If $\cut(y) \subset M$ denotes the cut-locus set of $y$ (i.e. the set of points $x \in M$ where $t \mapsto \exp_y(tx)$ loses minimality), then for $x \in M \setminus \cut(y)$, the gradient (or Levi-Civita covariant derivative) of $\di_y^2$ is given by 
\begin{align}   \label{eq:nabla_d2}
    \nabla \di_y^2(x) = - 2 \log_x(y) \, ,
\end{align}
where $\log_x = \exp_x^{-1}$.
% We consider a normal coordinate system centered at $x_0$, i.e. where the local chart around $x_0$ is given by the map $\log_{x_0}$.
% Thanks to the fact that, in this system of coordinates, the Christoffel symbols vanish at $x_0$, the Hessian of the squared distance can be computed as
The Hessian of the squared distance is defined as
\begin{equation}
     \nabla^2 \di_y^2(x_0)(v,w) = V(W \di_y^2)- (\nabla_VW)\di_y^2 \, ,\qquad v,w\in T_{x_0}M\,,
\end{equation}
where $V$ and $W$ are any extension to a vector field of $v$ and $w$, respectively. In particular, as a quadratic form, we have
\begin{equation}
    \nabla^2 \di_y^2(x_0)(v,v) = \frac{d^2}{dt^2}\bigg|_{t=0}  \di_y^2(\exp_{x_0}(tv)) ,\qquad v\in T_{x_0}M\,.
\end{equation}
% This is the proof of the formula above for ourselves.
% \begin{align}
%     \frac{d^2}{dt^2}\bigg|_{t=0}  \di_y^2(\exp_{x_0}(tv))               &=\frac{d}{dt}\bigg|_{t=0}\langle \nabla \di^2_y,\dot\gamma(t)\rangle
%       =\langle \nabla_{\dot\gamma}\nabla \di^2_y,\dot\gamma(t)\rangle|_{t=0}\\
%       &=\langle \nabla_v\nabla \di^2_y,v\rangle = \langle \nabla_V \nabla \di^2_y,V \rangle|_{x_0}
% \end{align}
% \begin{align}
%     \frac12 \nabla^2 \di_y^2(x_0) = - \text{d}_x \log_x(y)\big|_{x=x_0} \, , 
% \end{align}
% where d$_x$ denotes the differential at $x$. Therefore, by the properties of the $\log$ and $\exp$, one can show that (see e.g. \cite{Pennec:2017}), in these coordinates, 
Thus, using \cite[Theorem~3.8]{Loeper:2009}, we deduce (see also e.g. \cite{Pennec:2017}),
\begin{align}   \label{eq:hessian_squared_distance}
    \frac12 \nabla^2 \di_y^2(x_0) = \id - \frac13R_y(x_0) + O \big( \di_\g^4(x_0,y) \big) \, ,
\end{align}
where $R_y(x_0)$ denotes the symmetric $(0,2)$-tensor given by
\begin{align}
    (v,w) \in T_{x_0}M \times T_{x_0}M
        \mapsto
    \g_{x_0}
    \big(
        \rie
        \left(
            \log_{x_0}(y), v
        \right)
        \log_{x_0}(y)
    , 
        w 
    \big) \, .
\end{align}
% Proof of \eqref{eq:hessian_squared_distance}, using \cite[Theorem~3.8]{Loeper:2009}. We let $y$ converge to $x_0$ along the curve $\eta(s)=\exp_{x_0}(sw)$, with $w\in T_{x_0}M$ and we study the expansion at $s=0$ of the function
% \begin{equation}
%     \frac{d^2}{dt^2}\bigg|_{t=0}\di^2(\exp_{x_0}(tv),\exp_{x_0}(sw))\,.
% \end{equation}
% Notice that if $w=v$, then $\di^2(\exp_{x_0}(tv),\exp_{x_0}(sw))=|t-s|^2$ and the original function is identically zero. Otherwise, consider $w$ unit and orthogonal to $v$.
% 
% 
Notice that \eqref{eq:hessian_squared_distance} is a statement between bilinear forms on the finite-dimensional vector space $T_{x_0}M$, thus the norm of the tensor 
\begin{align} 
    \frac12 \nabla^2 \di_y^2(x_0) - \id + \frac13R_y(x_0) \, ,
\end{align}
goes to $0$ as $y\to x_0$ (with order $4$), for any choice of operator norm on $T_{x_0}M\times T_{x_0}M$. From the symmetry of the Riemann tensor, we know that the tensor $R_y(x_0)$ is symmetric. 
Therefore, for every reference frame of $T_{x_0}M$, the matrix representation of $R_y(x_0)$ is self-adjoint, hence it is diagonalizable with orthogonal eigenspaces. We choose $\{e_1, \dots, e_n \} \subset T_{x_0}M$ to be an orthonormal basis of eigenvectors of $R_y(x_0)$.

\begin{rmk}
Observe that for $v\in T_{x_0}M$, such that $\|v\|=1$ and $\langle v , \log_{x_0}(y) \rangle = 0$, $R_y(x_0)(v,v)$ is the sectional curvature  of the plane generated by $v$ and $\log_{x_0}(y)$. Thus, the choice of $Q_\eps(x_0)$ encodes information of the sectional curvature of $M$ at $x_0$. This could present a possible issue for extensions of this technique to non-smooth settings.
\end{rmk}

We claim that, with this particular choice, the set of $\frac12$-midpoints between $Q_\eps(x_0)$ and $T^\lambda(Q_\eps(x_0))$, where $T^\lambda$ is defined in \eqref{eq:transport_map}, is quantitatively close (in measure) to $T_{\frac12}^\lambda(Q_\eps(x_0))$, for $\lambda$ sufficiently small. Recall that the $t$-midpoints between two sets $A,B$ are given by
\begin{align}
    M_t(A,B):=
    \left\{
        y=\exp_x(tv) 
            \suchthat 
        x \in A \, , \ \exp_x v \in B 
    \right\} \, .
\end{align}
% In the next statement, the constants are implicitly assumed to be only depending on $K$, $\delta$, $N$, and the potential $V$.
\begin{prop}[Control for the measure of the midpoints]
    \label{prop:control}
Let $x_0 \in M$, $v_0 \in T_{x_0}M$ be a point satisfying \eqref{eq:ric_absurd}. Recall the definition of $T^\lambda$ in \eqref{eq:transport_map} and set $y_0 := T^\lambda(x_0)$. Let $\{e_1, \dots, e_n \}$ to be an orthonormal basis of eigenvectors of $R_{y_0}(x_0)$ and for $\eps \in (0,1)$, let $Q_\eps(x_0)$ be the corresponding cube, as described in Section~\ref{sec:infinitesimal}. Then there exists $\lambda_2 \in (0,\bar\lambda)$, where $\bar\lambda$ is defined in \eqref{eq:refined_ineq}, and $\bar\eps\in (0,1)$, such that, whenever $\lambda \leq \lambda_2$ and $\eps\leq\bar\eps$, we have that 
\begin{align}
    \wei
    \left(  
        M_{\frac12} \big( Q_\eps(x_0) , T^\lambda(Q_\eps(x_0)) \big) 
    \right)^\frac1N
        \leq 
    \big( 1 + C \big( \lambda^4 + \eps \big) \big)
        \wei 
        \left( 
            T_{\frac12}^\lambda \big( Q_\eps(x_0) \big) 
        \right)^\frac1N \, ,
\end{align}
where $C \in \R_+$ is a constant that does not depend on $\lambda$ and $\eps$.
\end{prop}

\begin{rmk}
Recalling Remark \ref{rmk:measurability_Mt}, since $\exp_x$ is a local diffeomorphism around $x$, $M_t(Q_\eps(x_0) , T^\lambda(Q_\eps(x_0)) )$ is measurable for any $t\in[0,1]$.
\end{rmk}
% A convention about the notation: whenever $F:M \to M$ is a smooth function, then we denote by $\nabla F$ the $(0,2)$-tensor defined via the formula
% \begin{align}
%     \nabla F(x)(v,w) := \g_x 
%         \left(
%             \de F(x) v , w
%         \right)
%             \, , \quad \forall x \in M \, , \forall v,w \in T_xM \, .
% \end{align}
In order to prove this proposition, we need the following preliminary result.
\begin{lem}     \label{lem:gradients}
    Let $\psi \in C^3(M)$ such that $T(x):= \exp_x(\nabla \psi(x))$ is optimal and define $T_t(x):= \exp_x(t \nabla \psi(x))$. Then,
    \begin{enumerate}
        \item We have that $\nabla \psi(x)= -\nabla \di_{y}^2(x)/2$, where $y = T(x)$.
        \item
        % In normal coordinates centered in $x_0$, NON SERVE
        For every $x \in M$, and for every $t \in [0,1]$, in normal coordinates centered at $x$, we have that
    \begin{align}
        \emph{d}_x T_t = \nabla^2
        \bigg(
            \frac12 \di_{z_t}^2+ t \psi
        \bigg)(x)
            \, , \quad z_t := T_t(x) \, .
    \end{align}
    % In realta', l'uguaglianza di cui sopra si deve intendere tra tensori (1,1) oppure (0,2), a meno di isomorfismi -- addirittura anche solo in T_{x_0}M.
    \end{enumerate}
\end{lem}
The proof of this lemma closely follows the one of \cite[Proposition~4.1]{Cordero--Erasquin-McCann-Schmuckenschlager:2001}, for completeness we report here a concise proof of this fact.

\begin{proof} 
The proof of $(1)$ follows from \cite[Lemma~3.3(b)]{Cordero--Erasquin-McCann-Schmuckenschlager:2001}. We now prove $(2)$: without loss of generality, we can prove the claimed equality for $t=1$, the general case follows by suitable rescaling.
We fix $x \in M$ and set $y := T(x)=\exp_x(\nabla \psi(x))$. Define 
\begin{align}
    h(z):= \frac12 \di_y^2(z) + \psi(z)
        \, , \quad 
    \forall z \in M \, .
\end{align}
Let $u \in T_x M$ and set, for any $s\in [0,1]$, $x_s:= \exp_x(su)$, $y_s := T(x_s)=\exp_{x_s}(\nabla\psi(x_s))$, where in particular $x_0=x$, $y_0=y$ (see Figure~\ref{fig:perturb}).
\begin{figure}[ht]
    \centering
    \captionsetup{justification=centering}
        \includegraphics[scale=1]{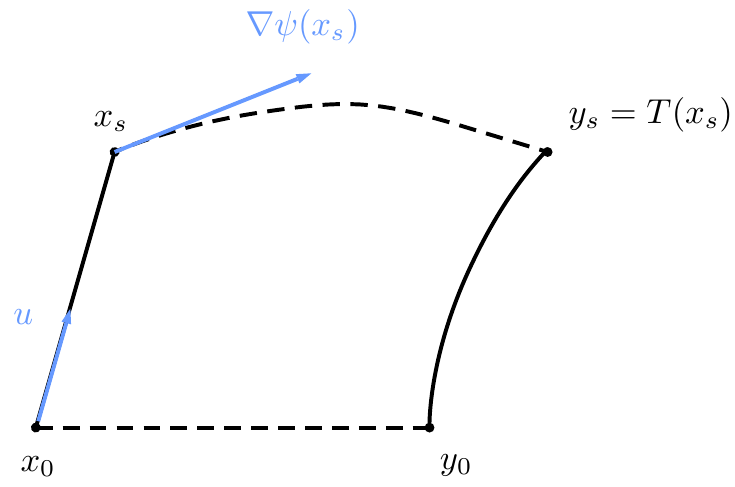}
    \caption{Picture of $s\mapsto x_s, y_s$ in normal coordinates centered at $x$.}
    \label{fig:perturb}
\end{figure} We also define
\begin{align}   \label{eq:defFG}
    F(z,v):= \exp_z(v)
        \, , \quad 
    G(z) := - \frac12 \nabla 
        \di_y^2 (z)\,,\qquad\forall\, v\in T_zM,\,z\in M
\end{align}
and note that by construction
\begin{align}
\label{eq:propFG}
    F(z,G(z)) = y
        \, ,\quad 
    \forall \,z \in M \, .
\end{align}
We also define $w_s:= \nabla h(x_s)$ and observe that $w_0=0$, using point $(1)$ of the statement, and by construction $y_s=F(x_s,\nabla\psi(x_s))=F(x_s, G(x_s) + w_s)$. Therefore, by computing the time derivative in $s$, at $s=0$, we find that
\begin{align}
    \dot y_s|_{s=0} &= \frac{\de}{\de s} \bigg|_{s=0}
    F(x_s, G(x_s) + w_s)=   
    \de_x F(\cdot,G(\cdot))\,\dot x_s|_{s=0} + \de_{G(x)}          F(x,\cdot)\,\dot w_s|_{s=0} \\
&=    
    \de_{G(x)} F(x, \cdot)\, \dot w_s|_{s=0} 
= 
    \big( \de_{G(x)} \exp_{x} \big)\, \dot w_s|_{s=0} \, ,
\end{align}
where we used that $w_0 = 0$ and  $\de_xF(\cdot, G(\cdot)) = 0$, which follows from \eqref{eq:propFG}. Working in normal coordinates centered at $x$, the exponential map (with base point $x$) is the identity, thus $\de_{G(x)} \exp_x = \id$. Consequently, we conclude that
 \begin{align}
     \de T_x(u) 
        =
    \frac{\de}{\de s} \bigg|_{s=0} T(x_s)
        =
    \dot y_s|_{s=0}  = \dot w_s|_{s=0}  
        = \nabla^2 h(x)(u) \, ,
 \end{align}
 for every $u \in T_{x}M$, which concludes the proof.
\end{proof}

We are ready to prove Proposition~\ref{prop:control}.
\begin{proof}[Proof of Proposition~\ref{prop:control}]
    We set $A_\eps:= Q_\eps(x_0) $ and $B_\eps:= T^\lambda(Q_\eps(x_0))$. 
    The idea is that, for if $T^\lambda$ were linear and \eqref{eq:hessian_squared_distance} were without the fourth--order error in the distance, we would be able to prove an exact set equality between $M_{\frac12} \big( A_\eps , B_\eps \big)$ and $T^\lambda_{\frac12} \big( A_\eps \big)$. Therefore, we linearize around $x_0$ and quantitatively study the error of this procedure.

    From now on, we assume to work in normal coordinates centered at $x_0$
% (meaning that $x_0=0$ in coordinates)
, in a neighborhood $U \subset M$.  % This means that all the gradients appearing in the remaining part of the proof has to be intended in coordinates with respect to this reference system. 
With slight abuse of notation, we do not change names of quantities when written in coordinates.

Recall that $\log_x=\exp_x^{-1}$. We define the map $F(\cdot,\cdot)$ as
    \begin{align}
        F(x,y) := \exp_x 
        \left(
            \frac12 \log_x y
        \right)
            = \exp_x 
        \left(
            -\frac12 \nabla \left( \frac12 \di_y^2(\cdot) \right)(x)
        \right)
            \, , \quad 
        \text{for} \ 
            x, y \in M \,  ,
    \end{align}
    where the second equality follows from \eqref{eq:nabla_d2}, and observe that, by definition of midpoints,
    \begin{align}
        M_{\frac12} \big( A_\eps , B_\eps \big) 
            = F 
        \big(
            A_\eps \times B_\eps
        \big) \, .
    \end{align}
    
    \noindent
    \underline{Step 1: expansions around $x_0$}. \
    A simple computation shows that, in normal coordinates centered at $x_0$, the differential of $F$ with respect to the second variable reads
    \begin{align}   \label{eq:expansion_Fy}
        \de_y F(x_0,\cdot) = \frac12 \id
            \, , \quad \forall y \in U \, .
    \end{align}
    To compute the differential with respect to the first variable, we set $\tilde \psi:= -\frac12 \di_{y_0}^2$ and note that $F(x,y_0) = \tilde T_{\frac12}(x) := \exp_x \Big( \nabla \tilde \psi(x) /2 \Big)$. Therefore, thanks to Lemma~\ref{lem:gradients}, we obtain that
    \begin{align}   \label{eq:expansion_Fx}
        \de_{x_0} F(\cdot,y_0) 
        % = \nabla T_{\frac12} (0) 
        = \nabla^2 
        \left(
            \frac12 \di_{\tilde T_{\frac12}(x_0)}^2 + \frac12 \tilde\psi \right) (x_0)
        = \nabla^2 
        \left(
            \frac12 \di_{z_0}^2 - \frac14 \di_{y_0}^2    
        \right) (x_0)
        \, ,
    \end{align}
where $\tilde T_{\frac12}(x_0) = T^\lambda_{\frac12}(x_0) =: z_0$. Note that, by $\log_{x_0}(z_0) = \frac12 \log_{x
_0}(y_0)$ and and the homogeneity of the Riemann tensor, we get that $R_{z_0}(x_0) = \frac14 R_{y_0}(x_0)$. 
From the expansion \eqref{eq:hessian_squared_distance}, we then find that
\begin{align}
    \de_{x_0} F(\cdot,y_0) 
        &= \id -\frac13 R_{z_0}(x_0) - \frac12
            \left(
                \id - \frac13 R_{y_0}(x_0)
            \right)
        + O(\lambda^4) 
\\        
        &=\frac12 
        \left(
            \id + \frac16 R_{y_0}(x_0)
        \right)
        + O(\lambda^4) \, .
\end{align}
% Here and from now on, $O(\cdot)$ denotes a function with sublinear growth, which does not depend on $\lambda$ and $\eps$. 
Similarly, by $R_{T^\lambda_t(x_0)}(x_0) = t^2 R_{y_0}(x_0)$, and applying Lemma \ref{lem:gradients}, we find that
\begin{align}   \label{eq:expansion_Tt}
    \de_{x_0} T^\lambda_t = \id - \frac{t^2}3 R_{y_0}(x_0) + t \nabla^2 \psi_\lambda(x_0) + O(\lambda^4)
        \, , \quad  \forall t \in [0,1] \, ,
\end{align}
as $\lambda\to 0$.
% Con $x_0$ invece che 0, ma e' pesante...
% Note that in this coordinates, $z_0= (x_0 + y_0)/2$.  Hence, using the differentiability of $F$ in $(x,y)= (x_0,y_0)$, from \eqref{eq:Taylor_Fy}, \eqref{eq:Taylor_Fx}, and \eqref{eq:Taylor_Tt}, we obtain the first order expansion, for $x, x' \in Q_\eps(x_0)$ -- recall that $\diam \big( Q_\eps(x_0) \big) = O(\eps^{\frac1n})$ -- 
% \begin{align*}
%     F(x,T(x')) 
%         &= F(x_0,y_0) + \nabla_x F(x_0,y_0) (x-x_0) + \nabla_y F(x_0,y_0) (T(x')-y_0) + O(\eps^2)
%     \\
%         &= z_0 + \frac12 
%             % \big(
%                  (T(x')-y_0)
%             % \big)
%     %     &= z_0 + \frac12 
%     %     \left(
%     %         \id - \frac13 R_{y_0}(x_0)
%     %     \right) (x'-x_0) 
%             + \frac12
%         \left(
%             \id   + \frac1{6} R_{y_0}(x_0) 
%         \right) (x-x_0)
%       + O(\lambda^4 \,  \eps^2) + O(\eps^2)
%     \\
%         &= z_0 + \frac12
%             \big(
%                 (x-x_0) + (x'-x_0)
%             \big)
%          - \frac1{12} R_{y_0}(x_0) + + O(\lambda^4 \,  \eps^2) +  O(\eps^2)
% \end{align*}
% as well as, for every $z \in U$,
We compute the Taylor expansion of the map 
\begin{equation}
    A_\eps \times A_\eps \ni ( x,x' ) \mapsto F ( x,T^\lambda( x' ) )
\end{equation}
at the point $( x_0,x_0 )$. 
Using the differentiability of $F$ in $(x,y)= (x_0,y_0)$, from \eqref{eq:expansion_Fy}, \eqref{eq:expansion_Fx}, and \eqref{eq:expansion_Tt} with $t=1$, for $x, x' \in A_\eps$ -- recall that $\diam \big( A_\eps \big) = O(\eps)$ -- in coordinates, as $\eps\to 0$, we obtain 
\begin{align*}
    F(x,T^\lambda(x')) 
        &= F(x_0,y_0) + \de_{x_0} F(\cdot,y_0) (x-x_0) + \de_{y_0} F(x_0,\cdot) (T^\lambda(x')-y_0) + O(\eps^2)
    \\
        &= z_0 + \frac12 
            % \big(
                 (T^\lambda(x')-y_0)
            % \big)
    %     &= z_0 + \frac12 
    %     \left(
    %         \id - \frac13 R_{y_0}(x_0)
    %     \right) (x'-x_0) 
            + \frac12
        \left(
            \id   + \frac1{6} R_{y_0}(x_0) 
        \right) (x-x_0)
      + O\big( \eps^2 + \eps \lambda^4  \big)
    \\  \label{eq:Taylor_F}
        &= z_0 + \frac12
            \left(
                \id - \frac13 R_{y_0}(x_0) + \nabla^2 \psi_\lambda(0)
            \right) (x'-x_0) 
         \\
            &\qquad\qquad\qquad\qquad\quad\,+\frac12 
         \left(
            \id + \frac1{6} R_{y_0}(x_0) 
        \right) (x-x_0)
            +  O\big( \eps^2 + \eps \lambda^4  \big).
\end{align*}
We remark that, to deduce the error terms in the first and third equality, we have performed a Taylor expansion of $F$ and exploited the fact that 
\begin{equation}
\label{eq:C2_norm_Tlambda}
    \|T^\lambda\|_{C^2(U)}\leq L \left( 1+\|T^1\|_{C^2(U)} \right),
\end{equation} 
with $L>0$ independent on $\lambda$. To prove \eqref{eq:C2_norm_Tlambda}, it is enough to use Lemma \ref{lem:gradients}, together with the expansion \eqref{eq:hessian_squared_distance}.
Similarly, using \eqref{eq:expansion_Tt} with $t=\frac12$, we have that, for any $x'' \in A_\eps$,
\begin{align}
    T_{\frac12}^\lambda(x'') 
        &= z_0 + \de_{x_0} T_{\frac12}^\lambda ( x''-x_0 ) + O(\eps^2)
    \\  \label{eq:Taylor_T12}
        & = z_0 + 
        \left(
            \id - \frac1{12}R_{y_0}(x_0) + \frac12 \nabla^2 \psi_\lambda(x_0)
        \right) ( x'' - x_0 )
            +  O\big( \eps^2 + \eps \lambda^4  \big) \, .
\end{align}
Taking into account condition \eqref{eq:conditions_psi}, we can then write 
\begin{align}   
\begin{aligned}
\label{eq:linearization_M123}
    F(x,T^\lambda(x'))  
        &= z_0 + \frac12 
        \Big(
            M_1 (x-x_0) + M_2 (x'-x_0)
        \Big) 
             +  O\big( \eps^2+ \eps \lambda^4  \big) \, ,
    \\ T_{\frac12}^\lambda(x'') 
        &= z_0 + M_3 (x''-x_0)   
            +  O\big( \eps^2 + \eps \lambda^4  \big) \, ,
\end{aligned}
\end{align}
for $x$, $x'$, $x'' \in A_\eps$, where the linear operators $M_i$ are given by
% \todo{Attenzione: qui non c'\`e $\alpha_0$ ma $\alpha_0\lambda$, perch\'e in Sec \ref{sec:choice_kant} abbiamo fissato $\alpha_0$ per $\psi$ e poi definito $\psi_\lambda$. -- Lore: Ne sai una piu' del diavolo Tommy, correggo!}
\begin{gather}
\label{eq:defM1M2}
    M_1 := \id + \frac16 R_{y_0}(x_0)
        \, , \quad 
    M_2:= (1+\lambda\alpha_0) \id - \frac13 R_{y_0}(x_0)
        \, , 
\\
    \label{eq:defM3}
    M_3 := \bigg( 1+\lambda\frac{\alpha_0}2 \bigg) \id - \frac1{12} R_{y_0}(x_0)
        = \frac12 (M_1 + M_2)
        \, .
\end{gather}

\noindent 
\underline{Step 2: solution to the linear problem}. \ 
We set $T_{\frac12, \text{lin}}^\lambda(x) := z_0 + M_3 (x-x_0)$. We claim 
\begin{align}   \label{eq:dist_M12_linT}
    M_{\frac12}
        \left(
            A_\eps , B_\eps
        \right)
\subset 
    \tB_{\rho_{\eps,\lambda}} 
    \Big(
        T_{\frac12, \text{lin}}^\lambda(A_\eps)
    \Big)
    % \text{dist}
    % \Big(
    %     M_{\frac12}
    %     \left(
    %         A_\eps , B_\eps
    %     \right)
    %         , 
    %     T_{\frac12, \text{lin}}^\lambda(A_\eps)
    % \Big)
    %     = O\big( \eps^2 + \eps  \lambda^4  \big) \, .
        \, , \quad 
    \rho_{\eps,\lambda} := C \big( \eps^2 + \eps \lambda^4 \big)
        \, , \quad C \in \R_+ \, ,
\end{align}
where $\tB_\rho(D)$ denotes the Euclidean $\rho$-enlargement of a set $D$, in coordinates.
It suffices to prove that, for every $x$, $x' \in A_\eps$, we can solve the problem
\begin{align}   \label{eq:linear_problem}
    \text{find} 
        \ \,
            x'' \in A_\eps \ \, \text{such that} 
        \quad  
    M_3 x'' = \frac12 
        \left(
            M_1 x + M_2 x'
        \right) 
    \, .
\end{align}
Indeed, fix $z\in M_{\frac12}(A_\eps,B_\eps)$, then by definition $z=F(x,T^\lambda(x'))$, for some $x,x'\in A_\eps$. Let $x''\in A_\eps$ solving \eqref{eq:linear_problem}, then, thanks to \eqref{eq:linearization_M123},
\begin{equation}
    z = F(x,T^\lambda(x')) = T_{\frac12, \text{lin}}^\lambda(x'') + O\big( \eps^2 + \eps \lambda^4 \big) \, ,
\end{equation}
thus proving \eqref{eq:dist_M12_linT}. In order to prove claim \eqref{eq:linear_problem}, note there exists $\lambda_2$ sufficiently small such that, for $\lambda \leq  \lambda_2$, the matrices $M_i$ are positive definite, since $\| R_{y_0}(x_0) \|_{\text{HS}} \leq C \lambda^2$.
% \todo{Stesso discorso di sopra: $\alpha_0$ \`e fissato e poi moltiplicato per $\lambda$ (e quindi \`e comunque piccolo di ordine $\lambda$) -- Corretto/cancellato} by \eqref{eq:conditions_psi}. 
In particular, the matrix $M_3$ is invertible. It follows that the problem \eqref{eq:linear_problem} is solved as soon as we can ensure that
\begin{align}   \label{eq:final_proof}
    \forall x, \, x' \in A_\eps 
        \, , \quad 
    M_3^{-1}
    \left( 
        \frac12 
            \left(
                M_1 x + M_2 x'
            \right) 
    \right) \in A_\eps
    \, .
\end{align}
This is a consequence of the fact that $A_\eps$ is a cube of eigenvectors of $R_{y_0}(x_0)$. Indeed, let $\{\mu_i^1 \}_i$, $\{\mu_i^2 \}_i$, $\{\mu_i^3 \}_i$ be the corresponding (positive) eigenvalues associated with $\{e_i\}_i$ of the matrices $M_1$, $M_2$, $M_3$, respectively (note that they share the same eigenspaces). By definition of normal coordinates and $A_\eps$, every $x$, $x' \in A_\eps$ can be written as
\begin{align}
    x = x_0 + \sum_{i=1}^n x_i e_i
        \, , \quad 
    x' = x_0 + \sum_{i=1}^n x_i' e_i
        \, , \quad 
    x_j \, , \, x_j'  \in [0,\eps]
        \, , \quad 
    \forall j = 1 , \dots , n \, .
\end{align}
Therefore, we have that
\begin{align*}
   M_3^{-1}
    \bigg( 
        \frac12 
            \left(
                M_1 x+ M_2 x'
            \right) 
    \bigg) - x_0
        &=
        M_3^{-1}
    \bigg(
        \frac12 \sum_{i=1}^n
            \big( 
                \mu_i^1 x_i + \mu_i^2 x_i'
            \big) e_i
    \bigg) 
% \\
        % &
        =
    \sum_{i=1}^n    
    \bigg(
        \frac{\mu_i^1 x_i + \mu_i^2 x_i'}{2\mu_i^3}  
    \bigg) e_i\, .
\end{align*}
Recall that $M_3 = (M_1 + M_2)/2$, which in particular implies that, for every $i=1, \dots, n$, 
\begin{align}
    \frac{\mu_i^1 + \mu_i^2}{2\mu_i^3} = 1 
        \quad \Longrightarrow \quad 
        \frac{\mu_i^1  x_i + \mu_i^2 x_i'}{2\mu_i^3} 
    \in [0, \eps] \, .
\end{align}
This concludes the proof of \eqref{eq:final_proof}, and hence of \eqref{eq:linear_problem}.

\smallskip
\noindent
\underline{Step 3: measure comparison}. \ 
In the remaining part of the proof, the constant $C \in \R_+$ does not depend on $\lambda$ and $\eps$, and might change line by line. Taking advantage of \eqref{eq:expansion_Tt} once again, recalling the definition of $M_3$ in \eqref{eq:defM3}, we infer 
\begin{align}   \label{eq:expansion_gradT_t}
    \de_x T_{\frac12}^\lambda = M_3 + O \big( \lambda^4 + \eps \big) 
        \, , \quad \forall x \in A_\eps \, .
\end{align}
By the Jacobi's formula for the derivative of the determinant, we deduce that
% \todo{N.B.: Qui non c'\`e l'uguaglianza con la $C$, ma piuttosto con l'$O(\eps+\lambda^4)$ e analogo sotto. -- Lore: hai ragionissima, cambiato.}
\begin{align}
\label{eq:expansion_det_dT_t}
    \det \left( \de_x T_{\frac12}^\lambda \right) = \det M_3 
    \left(
        1 + O \big( \lambda^4 + \eps \big)     
    \right) \, , \quad \forall x \in A_\eps \, .
\end{align}
% % notice that $C=\tr ( M_3^{-1} O( \lambda^4 + \eps ) ) = O ( \lambda^4 + \eps )$ 
% Hence, we can estimate the measure of $T_{\frac12}(A_\eps)$ as 
% \begin{align}
% \begin{aligned}
%     \wei
%     \left(
%         T_{\frac12}^\lambda(A_\eps)
%     \right)
%     %     &= 
%     % \int_{A_\eps} \de
%     % \Big(
%     %     T_{\frac12}^{-1}
%     % \Big)_{\#} \wei
%         &=
%     \int_{A_\eps} 
%     \left|
%         \det \Big( \de T_{\frac12}^\lambda \Big)
%     \right| \de \wei 
%         =\int_{A_\eps} \left| \det M_3 \right|
%         \left(
%             1+ O\big( \lambda^4 + \eps         \big) 
%         \right) \de \wei
% \\
%         &= \left(
%             1+ O\big( \lambda^4 + \eps \big) 
%         \right) \wei \left( T_{\frac12,\text{lin}}^\lambda(A_\eps) \right) \, .
% \end{aligned}
% \end{align}
%\todo[inline]{Nota che basta $C^1$ tra l'altro. Inoltre, $f \in C^2(\exp_{x_0}^{-1}(U))$ invece che $C^2(U)$?}
Let $f\in C^2$ be the density of $\wei$ in coordinates and set $f^\lambda := f \circ T_{\frac12}^\lambda$, $f_{\lin}^\lambda := f \circ T_{\frac12, \text{lin}}^\lambda$. We expand $f$ at $T_{\frac12,\text{lin}}^\lambda(x)$ for $x\in A_\eps$, and evaluate the expansion at $T_{\frac12}^\lambda(x)$, obtaining
\begin{equation}
\label{eq:expansion_flambda}
    f^\lambda(x) = f_\lin^\lambda(x) \Big( 1 + O \Big( \big\| T_{\frac12}^\lambda - T_{\frac12,\text{lin}}^\lambda \big\|_{L^\infty(A_\eps)} \Big) \Big) = f_\lin^\lambda(x)\Big( 1 + O \big( \eps^2+\eps \lambda^4 \big) \Big) \, , 
\end{equation}
where we used \eqref{eq:linearization_M123} and the fact that $f$ is of the form $e^{-\tilde V}$, for some $\tilde V \in C^2$ depending on $V$ and the metric $\g$. Thus, by \eqref{eq:expansion_det_dT_t} and \eqref{eq:expansion_flambda}, we have that
\begin{align}
\label{eq:T--Tlin_measure_comp}
\begin{aligned}
    \wei
    \left(
        T_{\frac12}^\lambda(A_\eps)
    \right)
    %     &= 
    % \int_{A_\eps} \de
    % \Big(
    %     T_{\frac12}^{-1}
    % \Big)_{\#} \wei
        &=
    \int_{A_\eps} 
    \left|
        \det \Big( \de T_{\frac12}^\lambda \Big)
    \right| f^\lambda(x)\,\de x 
\\
        &=\int_{A_\eps} \left| \det M_3 \right|
        \left(
            1+ O\big( \lambda^4 + \eps \big) 
        \right)
         f_\lin^\lambda(x) \Big( 1 + O\big( \eps^2 + \eps \lambda^4 \big) \Big)\de x
\\
        &= \left(
            1+ O\big( \lambda^4 + \eps \big) 
        \right) \wei \left( T_{\frac12,\text{lin}}^\lambda(A_\eps) \right) \, .
\end{aligned}
\end{align}
Moreover, by means of basic properties of the linear maps, it is easy to see that
\begin{align}   
\label{eq:ball_inclusion}
    \tB_{\rho_{\eps,\lambda}} 
    \Big(
        T_{\frac12, \text{lin}}^\lambda(A_\eps)
    \Big)
        \subset 
    T_{\frac12, \text{lin}}^\lambda
        \left(
            \tB_{c \rho_{\eps,\lambda}}(A_\eps)
        \right)
        \, , \quad  c \in \R_+ \, ,
\end{align}
where $\rho_{\eps,\lambda}$ is defined in \eqref{eq:dist_M12_linT} and $c$ is an upper bound for $\| M_3 \|_{\text{op}}$ uniform in $\lambda$.
% In particular, we can bound from above the measure of the left-hand side as
% \begin{align}
%     \wei 
%     \left(
%          \tB_{\rho_{\eps,\lambda}} 
%     \Big(
%         T_{\frac12, \text{lin}}^\lambda(A_\eps)
%     \Big)
%     \right)
% = 
%     \wei 
%     \left(
%     T_{\frac12, \text{lin}}^\lambda
%         \left(
%             \tB_{c \rho_{\eps,\lambda}}(A_\eps)
%         \right)
%     \right)
% =
%     | \det M_3 | ....
% \end{align}
The next step consists of studying the ratio between the measure of the image, via the affine map $T_{\frac12, \text{lin}}^\lambda$, of the set $A_\eps$ and its enlargement. 
Denote by $\wei_\lin \in \cM_+(U)$ the measure with density $f_\lin^\lambda$ in coordinates.
% Denote by $f:= e^{-V}$ the density of $\wei$ with respect to the volume $\vol$\todo{Qui si potrebbe considerare direttamente $f$ densit\'a di $\wei$ in coordinate. Vd. sopra}, and set $f_{\lin}^\lambda := f \circ T_{\frac12, \text{lin}}^\lambda$, $\wei^\lambda := f^\lambda \vol$. 
Arguing as in \eqref{eq:T--Tlin_measure_comp}, we find that
\begin{align}   
    \frac
    {
    \wei
    \left(
    T_{\frac12, \text{lin}}^\lambda
        \left(
            \tB_{c \rho_{\eps,\lambda}}(A_\eps)
        \right)
    \right)
    }
    {
        \wei
        \left(
            T_{\frac12, \text{lin}}^\lambda(A_\eps)
        \right)
    }
        &=
    \frac
    {
    \wei_\lin
    \left(
        \tB_{c\rho_{\eps,\lambda}}     
            (A_\eps) 
    \right)
    }
    {
    \wei_\lin(A_\eps)
    }
      =
    \frac
    {
    \Leb^n
    \left(
         \tB_{c\rho_{\eps,\lambda}} (A_\eps) 
    \right)
    }
    {
        \Leb^n \big( A_\eps \big)
    }
    \frac
    { 
    \fint_{ \tB_{c\rho_{\eps,\lambda}}(A_\eps) }
        f_\lin^\lambda(x) \de x
    }
    {
        \fint_{A_\eps} 
        f_\lin^\lambda(x) \de x
    } \, .
\end{align}
% {\color{blue}
% where $\tilde \tB_{c\rho_{\eps,\lambda}}$ and $\tilde A_\eps$ denotes the corresponding sets in coordinates, i.e. 
% \begin{align}
%     \tilde \tB_{c\rho_{\eps,\lambda}}
%         := \exp_{x_0}^{-1}
%             \left(
%                 \tB_{c\rho_{\eps,\lambda}}  (A_\eps)
%             \right)
% \tand 
%     \tilde A_\eps := \exp_{x_0}^{-1} (A_\eps) \, .
% \end{align}
% }
On the one hand, an application of the Minkowski--Steiner formula for convex bodies, together with $\Haus^{n-1}(A_\eps) = C \eps^{-1} \Leb^n(A_\eps)$, yields
\begin{align}
    \Leb^n
    \left(
         \tB_{c\rho_{\eps,\lambda}}(A_\eps) 
    \right)
\leq 
     \Big(
        1+ C \eps^{-1}\rho_{\eps,\lambda}
    \Big)
    \Leb^n \big(  A_\eps \big)
= 
    \Big(
        1+ C (\lambda^4 + \eps)
    \Big)
    \Leb^n \big(  A_\eps \big) 
    \, .
\end{align}
On the other hand, reasoning as in \eqref{eq:expansion_flambda}, we perform a Taylor expansion of $f_\lin^\lambda$ at $x_0$, obtaining the following two-sided bound:
\begin{align}
    \Big(
        1- C (\eps + \rho_{\eps,\lambda} )
    \Big)
        \leq 
    \frac{f_\lin^\lambda(x)}{f_\lin^\lambda(x_0)}
        \leq 
    \Big(
        1+ C (\eps + \rho_{\eps,\lambda} )
    \Big)
        \, , \quad \forall x \in \tB_{c\rho_{\eps,\lambda}}(A_\eps) \, .
\end{align}
% \begin{equation}
%     f^\lambda(y)=f(y)+\frac12\langle \nabla f(y), \nabla\psi(y)\rangle\lambda +O(\lambda^2) 
% \end{equation}
Therefore, using that $\eps + \rho_{\eps,\lambda} \leq 2 \eps + \lambda^4$, we obtain the upper bound
\begin{align}   \label{eq:estimate_measure_linear_map}
    \wei
    \left(
    T_{\frac12, \text{lin}}^\lambda
        \left(
            \tB_{c \rho_{\eps,\lambda}}(A_\eps)
        \right)
    \right)
\leq 
    \Big(
        1+ C (\lambda^4 + \eps )
    \Big)
    \wei
        \left(
            T_{\frac12, \text{lin}}^\lambda(A_\eps)
        \right) \, .
\end{align}
In conclusion, putting together \eqref{eq:dist_M12_linT},  \eqref{eq:T--Tlin_measure_comp}, \eqref{eq:ball_inclusion}, and \eqref{eq:estimate_measure_linear_map}, we get that
% \begin{align}
%     \wei 
%     \left(
%         M_{\frac12}(A_\eps,B_\eps)
%     \right)
%         &\leq
%     \wei
%     \left(
%         \tB_{\rho_{\eps,\lambda}} \Big( 
%             T_{\frac12, \text{lin}}^\lambda(A_\eps) 
%         \Big)
%     \right)
%         \leq 
%     \wei 
%     \left(
%         T_{\frac12, \text{lin}}^\lambda
%         \left(
%             \tB_{c \rho_{\eps,\lambda}}(A_\eps)
%         \right)
%     \right)
% \\
%     &\leq
%     \big( 1 + C \big( \lambda^4+ \eps 
%     \big) \big)
%         \wei 
%         \left(
%             T_{\frac12, \text{lin}}^\lambda(A_\eps)
%         \right)
% \\
%         &\leq
%      \big( 1 + C \big( \lambda^4+ \eps 
%     \big) \big)
%         \wei \left( T_{\frac12}^\lambda(A_\eps) \right) \, ,
% \end{align}
\begin{align}
    \wei 
    \left(
        M_{\frac12}(A_\eps,B_\eps)
    \right)
        &\stackrel{\eqref{eq:dist_M12_linT}}\leq
    \wei
    \left(
        \tB_{\rho_{\eps,\lambda}} \Big( 
            T_{\frac12, \text{lin}}^\lambda(A_\eps) 
        \Big)
    \right)
        \stackrel{\eqref{eq:ball_inclusion}}{\leq} 
    \wei 
    \left(
        T_{\frac12, \text{lin}}^\lambda
        \left(
            \tB_{c \rho_{\eps,\lambda}}(A_\eps)
        \right)
    \right)
\\
    &\stackrel{\eqref{eq:estimate_measure_linear_map}}{\leq}
    \big( 1 + C \big( \lambda^4+ \eps 
    \big) \big)
        \wei 
        \left(
            T_{\frac12, \text{lin}}^\lambda(A_\eps)
        \right)
\\
        &\stackrel{\eqref{eq:T--Tlin_measure_comp}}{\leq}
     \big( 1 + C \big( \lambda^4+ \eps 
    \big) \big)
        \wei \left( T_{\frac12}^\lambda(A_\eps) \right) \, .
\end{align}
We take the power $\frac1N$ on both side: up to changing the constant $C$ once again and  considering $\lambda$ and $\eps$ sufficiently small, we conclude the proof.
\end{proof}

\subsection{Proof of the main result}
We are ready to prove our main result, Theorem~\ref{thm:main}.
\begin{proof}[Proof of Theorem~\ref{thm:main}]
By contradiction, we assume there exists $x_0 \in M$, $v_0 \in T_{x_0}M$ such that \eqref{eq:ric_absurd} is satisfied, for some $\delta>0$. Consider $\lambda$ sufficiently small, and let $T^\lambda$ be as in \eqref{eq:transport_map} and $y_0 := T^\lambda(x_0)$. In addition, let $\{e_1, \dots, e_n \}$ be an orthonormal basis of eigenvectors of $R_{y_0}(x_0)$ and, for $\eps \in (0,\bar\eps)$, let $Q_\eps(x_0)$ be the corresponding cube, as described in Section~\ref{sec:infinitesimal}.
Thanks to Proposition~\ref{prop:upper_bound_hatD} and Proposition~\ref{prop:control}, choosing $\lambda\leq\min\{\lambda_1, \lambda_2\}$, we know 
\begin{gather}
\label{eq:final_D}
    \cD_{x_0,\lambda}\bigg( \frac12 \bigg) 
        \leq \tau_{K,N}^{\left( \frac12 \right)}(\lambda)
            \left(
                \cD_{x_0,\lambda}(0) +  \cD_{x_0,\lambda}(1)  
            \right)    
             - c \lambda^2 \, ,\\
\label{eq:final_MT}
\wei
    \left(  
        M_{\frac12} \big( Q_\eps(x_0) , T^\lambda(Q_\eps(x_0)) \big) 
    \right)^\frac1N
        \leq 
    \big( 1 + C \big( \lambda^4 + \eps \big) \big)
        \wei 
        \left( 
            T_{\frac12}^\lambda \big( Q_\eps(x_0) \big) 
        \right)^\frac1N \, , 
\end{gather}
where $\cD_{x_0,\lambda}$ is the volume distortion function associated with $\gamma^\lambda$, i.e. 
\begin{align}
    \cD_{x_0,\lambda}(t):=
\lim_{\eps\to 0}
    \frac{\wei(T_t^\lambda(Q_\eps(x_0)))^{\frac1N}}{\wei(Q_\eps(x_0))^{\frac1N}}
        \, , \quad t \in [0,1] \, .
\end{align}

We select the marginal sets $A_\eps:= Q_\eps(x_0)$, $B_\eps:=T^\lambda (Q_\eps(x_0)) $. Recalling the definition of $\Theta(A,B)$ in \eqref{eq:def_Theta}, we have that $\Theta(A_\eps,B_\eps)=\lambda+O(\eps)$, as $ \eps\to 0$, hence we deduce that
\begin{equation}
\label{eq:expansion_tau_Theta}
   \tau_{K,N}^{(\frac12)}(\lambda)= \tau_{K,N}^{(\frac12)}\left(\Theta(A_\eps,B_\eps)\right)+O(\eps)\,,\qquad \text{as} \ \eps\to 0\,.
\end{equation}
Combining \eqref{eq:expansion_tau_Theta} with \eqref{eq:final_D}, we infer that there exists $\eps_0=\eps_0(\lambda)$ and $\bar c>0$, such that, for every $\eps \leq \eps_0$, we have the inequality 
\begin{equation}
\label{eq:LPORTSK}
    \wei\big(T_{\frac12}^\lambda(A_\eps)\big)^{\frac1N}
\leq 
    \tau_{K,N}^{\left( \frac12 \right)}(\Theta(A_\eps,B_\eps))
            \left(
                \wei(A_\eps)^{\frac1N} +  \wei(B_\eps )^{\frac1N}  
            \right)    
            +\left(-\frac{c}{2} \lambda^2    
               +\bar c \eps \right)\wei(A_\eps)^{\frac1N}  \, ,
\end{equation}
where we have also used \eqref{eq:C2_norm_Tlambda} to estimate $\wei(B_\eps)$ in terms of $\wei(A_\eps)$. As a consequence, if $\eps \leq \min\{\eps_0, \lambda^4\}$, from \eqref{eq:final_MT} and \eqref{eq:LPORTSK}, we obtain that
\begin{align*}
    \wei
    \left(  
        M_{\frac12} \big( A_\eps , B_\eps \big) 
    \right)^{\frac1N}
&\leq 
    \tau_{K,N}^{\left( \frac12 \right)}\left(\Theta(A_\eps,B_\eps)\right)
        \left(
            \wei(A_\eps)^{\frac1N} +  \wei(B_\eps )^{\frac1N}  
        \right)\\
    &\qquad\qquad\qquad\,+\left(- \frac c2 \lambda^2+\bar c \lambda^4 \right) \wei(A_\eps)^{\frac1N}
            + 2C \lambda^4 
             \wei\big(T_{\frac12}^\lambda(A_\eps)\big)^{\frac1N}\\
&\leq 
    \tau_{K,N}^{\left( \frac12 \right)}\left(\Theta(A_\eps,B_\eps)\right)
        \left(
            \wei(A_\eps)^{\frac1N} +  \wei(B_\eps )^{\frac1N}  
        \right) \\
        &\qquad\qquad\qquad\,+ 
        \left(
            -\frac c2 \lambda^2 + 2RC \lambda^4+\bar c \lambda^4
        \right) \wei(A_\eps)^{\frac1N}
            \, ,
\end{align*}
where we have used that $\wei\big(T_{\frac12}^\lambda(A_\eps)\big) \leq R \wei(A_\eps)$, for some $R>0$, which is a consequence once again of \eqref{eq:C2_norm_Tlambda}. Therefore, if $\lambda$ was chosen in such a way that $\lambda^2<c(4RC+2\bar c)^{-1}$, then for any $\eps \leq \min\{\eps_0, \lambda^4\}$, we finally conclude that
\begin{align}
\label{eq:final_ineq}
      \wei
    \left(  
        M_{\frac12} \big( A_\eps , B_\eps \big) 
    \right) ^{\frac1N}
<
    \tau_{K,N}^{\left( \frac12 \right)}\left(\Theta(A_\eps,B_\eps)\right)
        \left(
            \wei(A_\eps)^{\frac1N} +  \wei(B_\eps )^{\frac1N}  
        \right) \, ,
\end{align}
which is a contradiction to the space satisfying $\bm(K,N)$. This concludes the proof.
\end{proof}

\begin{rmk}
The estimate of Proposition \ref{prop:upper_bound_hatD} gives \emph{at best} a negative error in $\lambda$ of order $2$. Therefore, in Proposition \ref{prop:control} an estimate with a second--order precision would not be enough as the two errors would compete without having a definitive sign in the limit. Thus, we need to push the estimate of Proposition \ref{prop:control} to a fourth--order precision, in order to conclude that the error term is \emph{definitively negative}. Remarkably the estimate of Proposition \ref{prop:control} does not involve a term of order $2$ in $\lambda$ and we were finally able to prove \eqref{eq:final_ineq} carefully choosing $\lambda$.
\end{rmk}

\bibliographystyle{alphaabbr}

\end{document}